\documentclass[10pt]{article}

\usepackage{graphicx}
\graphicspath{{../../BCfigures/} {BCfigures/} }
\usepackage[colorlinks, linkcolor=blue,  citecolor=blue, urlcolor=blue]{hyperref}%
\usepackage{color}
\usepackage{subfigure}
\usepackage{enumerate}
\usepackage{verbatim}
\usepackage{alltt}

\newcommand{\bq}{\begin{equation}}
\newcommand{\eq}{\end{equation}}
\newcommand{\R}{\mathbb{R}}

\usepackage{pdfsync}
\usepackage{amsmath}\multlinegap=0pt
\usepackage[latin1]{inputenc}
\usepackage{amsthm}
\usepackage{amssymb}
\usepackage[francais,english]{babel}
\numberwithin{equation}{section}
\theoremstyle{plain}
\newtheorem{thm}{Theorem}[section]

\newtheorem{prop}[thm]{Proposition}

\theoremstyle{definition}

\theoremstyle{remark}

\newcommand\N{{\mathbb N}}

\newcommand\pref[1]{(\ref{#1})}

\newcommand{\weakstarto}{\rightharpoonup^*}

\DeclareMathOperator{\argmin}{argmin}
\DeclareMathOperator{\spt}{spt}

\DeclareMathOperator{\diam}{diam}

\def\<#1,#2>{\left<#1,#2\right>}

\def\PP{{\cal P}}

\newcommand\Lag{{\cal L}}

\newcommand\tc{\widetilde{c}}
\newcommand\tmu{\widetilde{\mu}}
\newcommand\tnu{\widetilde{\nu}}

\begin{document}
\title{Numerical methods for matching for teams and Wasserstein barycenters}
\author{
	G. Carlier
		\thanks{\scriptsize CEREMADE, UMR CNRS 7534, Universit\'e Paris IX Dauphine, Pl. de Lattre de Tassigny, 75775 Paris Cedex 16, FRANCE, 
		\texttt{carlier@ceremade.dauphine.fr}},  
	A. Oberman	
		\thanks{\scriptsize Department of Mathematics and Statistics, McGill University, 805 Sherbrooke Street West, Montreal, CANADA, 
		\texttt{adam.oberman@mcgill.ca}}, 
	E. Oudet 
		\thanks{\scriptsize Laboratoire Jean Kuntzmann, Universit\'e Joseph Fourier, Tour IRMA, BP 53 51, rue des Math\'ematiques F-38041 Grenoble Cedex 9, FRANCE,
		\texttt{edouard.oudet@imag.fr}.}
}
\maketitle

\begin{abstract}
Equilibrium multi-population matching (matching for teams) is a problem from mathematical economics which is related to multi-marginal optimal transport. 
  A special but important case is the  Wasserstein barycenter problem, which has applications in image processing and statistics. Two algorithms are presented: a linear programming algorithm and an efficient  nonsmooth optimization algorithm, which applies in the case of the Wasserstein barycenters. The measures are approximated by discrete measures: convergence of the approximation is proved. Numerical results are presented which illustrate the efficiency of the algorithms.
\end{abstract}


\section{Introduction}

Optimal transport theory has received a lot of attention in the last decades and is now recognized as a powerful tool in PDEs, geometry, and functional inequalities (for which we refer to the monographs of Villani \cite{villani}-\cite{villani2}). Given  two Borel probability  measures $\mu_1, \mu_2$, on metric spaces $X_1$ and $X_2$, respectively, and a cost function $c\in C(X_1\times X_2, \R)$, the Monge-Kantorovich optimal transport problem consists in finding the cheapest way  to transport $\mu_1$ to $\mu_2$ for the cost $c$:
\begin{equation}\label{MK1}\tag{MK}
W_{c}(\mu_1, \mu_2):=\inf_{\gamma \in \Pi(\mu_1, \mu_2)} \int_{X_1\times X_2} c(x_1, x_2) \gamma(dx_1, dx_2)
\end{equation}
where  $\Pi(\mu_1, \mu_2)$ denotes the set of transport plans between $\mu_1$ and $\mu_2$,  i.e.\  the set of probability measures on $X_1 \times X_2$ having $\mu_1, \mu_2$, respectively as marginals. Since this problem is of linear programming type, under very mild assumptions (e.g.\ when $X_1$ and $X_2$ are compact),  the least transport cost $W_{c}(\mu_1, \mu_2)$ admits a dual expression given by the Kantorovich duality formula:
\begin{equation}\label{dualmkform}
W_{c}(\mu_1, \mu_2):=\sup_{\varphi \in C(X_1, \R)}  \left\{  \int_{X_1} \varphi(x_1) \mu_1(dx_1)+\int_{X_2} \varphi^c(x_2) \mu_2(dx_2)  \right\}
\end{equation}
where $\varphi^c$ denotes the $c$-transform of $\varphi$:
\[\varphi^c(x_2):=\min_{x_1\in X_1} \{c(x_1,x_2)-\varphi(x_1)\}.\] 
A  particularly important example is the quadratic case where $X_1=X_2=\R^d$,  $\mu_1$ and $\mu_2$ have finite second moments, and $c(x_1, x_2) = \vert x_1-x_2\vert^2$. This case was first solved by Brenier \cite{Brenier}, who proved that whenever $\mu_1$ is absolutely continuous, there is a unique optimal transport plan that is given by the gradient of a convex potential.   This important result  relates optimal transport to Monge-Amp\`ere equations.  
 We refer to~\cite{benamou2014numerical} and the references therein for numerical methods for optimal transport  based on the Monge-Amp\`ere equation.

 More generally, costs given by distances or  convex power of distances are important because they lead to the so-called Wasserstein distances. More precisely, when $X_1=X_2$ (a metric space with distance $d$) and $c(x_1, x_2)=d(x_1, x_2)^p$ for some $p\ge 1$, the value $W_c(\mu_1, \mu_2)$ in \eqref{MK1} is the $p$-power of the so-called $p$-Wasserstein distance $W_p(\mu_1, \mu_2)$ between $\mu_1$ and $\mu_2$:
\[W_p(\mu_1, \mu_2):=   \Big( \inf_{\gamma \in \Pi(\mu_1, \mu_2)} \int_{X_1\times X_1} d(x_1, x_2)^p \gamma(dx_1, dx_2)  \Big)^{1/p}.\]

\smallskip

In the present article, we are interested in solving numerically the following variant of the optimal transport problem  which allows for more than  two  marginals.  Given (compact metric, say) spaces $X_1, \dots,  X_I$, equipped with Borel probability measures $(\mu_1, \ldots \mu_I)\in \PP(X_1)\times \ldots \times \PP(X_I)$, a (compact metric) space $Z$, and cost functions $c_i \in C(X_i\times Z, \R)$, we look for a probability measure $\nu$ on $Z$ solving:
\begin{equation}\label{ss}
\inf_{\nu\in \PP(Z)} J(\nu):= \sum_{i=1}^I W_{c_i} (\mu_i, \nu).
\end{equation}
This problem was introduced in Carlier and Ekeland \cite{ce}  in the framework of multi-population matching equlibrium; we will shortly recall in section \ref{mft} the economic interpretation of \eqref{ss}.  Problem \pref{ss} is also a special  case of multi-marginal optimal transport (a variant of \pref{MK1} where more than two marginals are prescribed).
Multi-marginal optimal transport is currently an active research field: 
compared to (two marginals) optimal transport, 
there are fewer theoretical results, and the complexity of general multi-marginal optimal transport problems typically increases exponentially in the number of marginals.  Regarding the rapidly developing theory of multi-marginal optimal transport, we refer the reader to the recent papers by  Pass \cite{pass1}, \cite{pass2}, by Ghoussoub and coauthors \cite{ghoussoub1}, \cite{ghoussoub2} and the references therein for costs with special symmetry properties, motivated in particular by challenging computational issues in density functional theory in quantum physics. 

\smallskip

We now discuss a special, but important case of  \eqref{ss} which has a clear geometric interpretation. Let all the $X_i$'s and $Z$ coincide with $\R^d$, the measures $\mu_i$ have finite second moments, and the costs be quadratic (i.e. $c_i(x_i, z):=\lambda_i \vert x_i -z\vert^2$ for some weights $\lambda_i>0$, summing to $1$ without loss of generality). In this case \eqref{ss} takes the form:
\begin{equation}\label{ssbar}
\inf_{\nu\in \PP(\R^d)} J(\nu):= \sum_{i=1}^I \lambda_i W_2^2 (\mu_i, \nu)
\end{equation}
where $W_2$ denotes the $2$-Wasserstein distance. In analogy with the Euclidean case, a solution to \eqref{ssbar} will be called a Wasserstein barycenter of the measures $\mu_i$ with weights $\lambda_i$. Properties of Wasserstein barycenters were studied by Agueh and Carlier \cite{ac}. Wasserstein barycenters interpolate between the measures $\mu_i$; the idea of interpolating between points of a metric space by minimizing some weighted sum of squared distances goes back to the notion of Fr\'echet mean. The case $I=2$ is well-known.  Letting the weights $(\lambda, 1-\lambda)$ vary, one obtains the classical notion of McCann's interpolation \cite{mci} between two probability measures. This interpolating curve is also a geodesic for $W_2$,  and in their seminal paper \cite{bb} on the dynamic formulation of optimal transport, Benamou and Brenier gave a numerical scheme to compute this geodesic.  Finding barycenters between more than two measures is more complicated (barycenters are not associative as soon as the dimension $d$ of the ambient space is larger than $2$).   From a Partial Differential Equations viewpoint, this problem requires to solve a system of Monge-Amp\`ere equations, see \eqref{MAsystem}-\eqref{sumid} below.
 Interestingly, the Wasserstein barycenter  problem  recently found natural applications in image processing, see  Peyr\'e et al.\ \cite{peyre} and statistics, see Bigot and Klein \cite{bk}. Of course, there are lots of variants of the interpolating scheme given by the $W_2$-barycenter problem \eqref{ssbar} and in particular one can replace $W_2$ by $W_p$ for some $p\ge 1$ or even mix different powers of the distance. Slightly abusing the terminology, we will  sometimes refer to barycenters even for these variants and even for the general form of the problem \eqref{ss}.

\smallskip

In the discrete setting, the transportation problem is classical.  In fact, this problem motivated the historical development of optimization, by Kantorovich in 1939, working on Soviet railway transportation, and in the 1940's by Hitchcock \cite[Ch 21]{schrijver2003combinatorial}. 
The  ``assignment problem'' arises in case of integer values weights, it is 
a standard combinatorial optimization  problem which can be solved by the Hungarian algorithm \cite[Ch 21]{schrijver2003combinatorial}.
More generally, the ``transportation problem'' is a linear programming problem which arises
   when the weights are real-valued, it can be solved by the Hitchcock algorithm, \cite[Ch 21]{schrijver2003combinatorial}, or by modern commercial general linear programming software. Returning to the problem with continuous measures, it is natural to approximate the measures  by weighted sums of delta measures. In theory, the resulting problem can be solved using linear programming.   However, the number of variables in the linear programming problem is  quadratic in the number of variables used to represent the measures.   In the discrete setting, current optimization algorithms are limited to approximately several thousands of variables for each of the measures.  This problem size corresponds to  a fairly coarse approximation of a two dimensional continuous  measure.   In special cases, or using specific approximations, improvements are available, see \cite{papadakis2014optimal} for quadratic costs, and  for more references. 
For example, if each measure is represented by, for example,  $40^2 = 1600$ variables, the linear program has $40^4 = 2~560~ 000$ which is near the limitations of linear programming algorithms (we performed experiments using CVX \cite{cvx} and calling several academic and commercial optimization packages).   Enlarging the resolution of the measures quickly overwhelms the capabilities of the algorithms.

The problem~\eqref{ss} is even more challenging, since it involves multiple marginals and an additional unknown measure.   Resolving the barycenter measure on the full grid generally leads to an intractable problem (but as shown by Cuturi and Doucet \cite{cuturidoucet}, some well-chosen smooth approximation can be solved in an efficient way).  Our main contributions regarding numerical schemes for the general problem \eqref{ss} or the particular case of Wasserstein barycenters \eqref{ssbar} are as follows:

\begin{itemize}

\item 
We give a simple linear programming reformulation of \eqref{ss}
in \autoref{lpr} whose size is proportional to the number of marginals. Together with a localization result  that bounds the support of the unknown barycenter in \autoref{sec:localization}, one then obtains a tractable problem. We discretize the problem to arrive at a  finite dimensional linear programming problem in \autoref{sec:discretization}. We prove convergence, in the sense of weak convergence of measures, in \autoref{sec:convergence}.

\item 
Numerical results are presented in \autoref{sec:num}.  These illustrate the validity of this linear programming approach.  Barycenter problems with different costs are solved, as well as a matching for teams problem.

\item The second algorithm which is specialized to the case of Wasserstein barycenter measures \eqref{ssbar}, is described and illustrated in \autoref{duality}.   This problem uses the  dual formulation of the problem explained in \autoref{duality}, and special features of the quadratic cost.   The  efficient nonsmooth optimization algorithm is described in 
\autoref{sec:algobar}. Large size computational examples are presented (on grids of size $200^2$, and for measures resolved with 15000 points).  The examples include barycenter measures using up to five measures, and an example in texture synthesis in \autoref{sec:Num2}.

\end{itemize}

\section{Matching for teams and approximation}\label{mft}

 In this section,  following \cite{ce}, we first derive  the generalized barycenter problem~\eqref{ss} as an equivalent reformulation of an equilibrium problem for multi-population matching arising in economics. 
 Next, we study localization of the barycenter measure.  Then, we present an infinite dimensional linear programming reformulation of \eqref{ss}. 
 This is followed  by a  discretization of the measures, which results in a finite dimensional linear programming problem that is tractable for moderate problem sizes.  Finally, we address stability issues (in the sense of weak convergence of measures) when one approximates the measures $\mu_i$ by some (discrete) measures. 
       
\subsection{Variational characterization of matching equilibria}

The model of Carlier and Ekeland \cite{ce} deals with the equilibrium of a market for a quality good (e.g.\ house, school, hospitals, \dots).  
Producing the good requires assembling a team consisting of a buyer and a set of producers.
For instance, in the case of houses, the producers could be a plumber, an electrician and a mason.  
The quality good has a range of feasible qualities (location, surface, number of rooms, facilities etc), denoted by $Z$ which we assume to be a compact metric space.  

Each of the different populations (buyers, plumbers, electricians, masons...) is indexed by  $i\in \{1,\dots, I\}$.
The agents in each population are hetererogeneous, characterized by a certain type which affects their (quality dependent) cost function. 
For example, some masons are used to work with lower quality bricks, while other work with luxury stones, some electricians live quite far from the location of the house they work on, consumers differ in their tastes... To be precise, for each population $i$, we are given a compact metric space of types, $X_i$, and a continuous  cost function $c_i\in C(X_i\times Z, \R)$ with the interpretation that $c_i(x_i,z)$ is the cost for an agent of population $i$ with type $x_i$ to work in a team that produces good $z$.  The distribution of type $x_i$ in population $i$ is known and given by some Borel probability measure $\mu_i\in \PP(X_i)$.

\smallskip

The goal is to find an equilibrium production line $\nu \in \PP(Z)$ (together with a price system) which clears both the quality good and the labor market.  The equilibrium is described below, and as we shall see, it corresponds to the solution of the (generalized) barycenter measure problem \eqref{primalmft}. In this setting, one looks in particular for an equilibrium system of  monetary transfers (paid by the buyer to the producers). A system of transfers is a collection of continuous functions $\varphi_1,\ldots \varphi_I$: $Z \to \R$ where $\varphi_i(z)$ is the amount paid to $i$ by the other members of the team for producing $z$. An obvious equilibrium requirement is that teams are self-financed i.e.
\begin{equation}\label{selff}
\sum_{i=1}^I \varphi_i(z)=0, \; \forall z\in Z.
\end{equation} 
Given transfers, $ \varphi_1,\ldots \varphi_I$, an agent from population $i$ with type $x_i \in X_i$, gets a net minimal cost  given by the so-called $c_i$-transform of $\varphi_i$:
\begin{equation}\label{citransf}
\varphi_i^{c_i}(x_i):=\min_{z\in Z} \{c_i(x_i,z)-\varphi_i(z)\}.
\end{equation}
By construction,  $\varphi_i^{c_i}(x_i)+ \varphi_i(z)\le c_i(x_i,z)$, and since agents are rational, they choose cost minimizing qualities, i.e. a $z\in Z$ such that
\begin{equation}\label{minfiz}
\varphi_i^{c_i}(x_i)+ \varphi_i(z)= c_i(x_i,z).
\end{equation}
The final unknown is a collection of plans, $\gamma_i\in \PP(X_i\times Z)$, such that $\gamma_i(A_i\times A)$ represents the probability that an agent in population $i$ has a type in $A_i$,  and belongs to a team that produces a quality in $A$. At equilibrium, the first marginal of $\gamma_i$ should be $\mu_i$ (this is equilibrium on the $i$-th labor market) and the second marginal of $\gamma_i$ should not depend on $i$ (this is equilibrium on the quality good market), this common marginal represents the equilibrium quality line.  An equilibrium can then be formally defined. It consists  of a transfer system $(\varphi_1,\ldots \varphi_I)\in C(Z, \R)^I$, probability measures $\gamma_i\in \PP(X_i\times Z)$, and a probability measure $\nu\in \PP(Z)$, such that
\begin{itemize}
\item teams are self-financed i.e. \pref{selff} holds,
\item $\gamma_i \in \Pi(\mu_i, \nu)$ for $i=1, \ldots, I$ (equilibrium on the labor markets and on the good market),  
\item \pref{minfiz} holds on the support of $\gamma_i$ for $i=1, \ldots, I$,  (i.e.\ agents choose cost minimizing qualities). 
\end{itemize}

If an  equilibrium quality line, $\nu$, was known, then clearly the last two conditions above would imply that the plan $\gamma_i$ should be optimal for the Monge-Kantorovich problem:
\begin{equation}\label{wci}
W_{c_i}(\mu_i, \nu):=\inf_{\gamma \in \Pi(\mu_i, \nu)} \int_{X_i \times Z} c_i(x_i, z) \gamma(dx_i, dz). 
\end{equation}

In fact, it was proved in \cite{ce} that there is a purely variational characterization of equilibria, which is tightly related to the following convex problem 
\begin{equation}\label{primalmft}
\inf_{\nu\in \PP(Z)} J(\nu):= \sum_{i=1}^I W_{c_i} (\mu_i, \nu)
\end{equation}
and its dual (concave maximization) formulation (see \cite{ce} or \autoref{duality} for details on this duality)
\begin{equation}\label{dualmft}
\sup\left\{  \sum_{i=1}^{I}\int_{X_{i}}\varphi_{i}^{c_{i}}(x_i)%
\mu_{i}(dx_i)\mbox{  :  }\sum_{i=1}^{I}\varphi_{i}=0\right\}.
\end{equation}

\begin{thm}\label{caractvarmft}
$(\varphi_{i},\gamma_{i},\nu)$ is an equilibrium if and
only if:

\begin{itemize}

\item $\nu$ solves $\pref{primalmft}$,

\item the transfers $(\varphi_1,\ldots \varphi_I)$ solve \pref{dualmft},

\item for  $i=1,\ldots,I$, $\gamma_{i}$ solves the Monge-Kantorovich problem $W_{c_{i}}(\mu_{i},\nu)$.

\end{itemize}

\end{thm}

\subsection{Localization}\label{sec:localization}

As noted in \cite{ce}, the minimization problem \pref{primalmft}, which characterizes equilibrium quality lines, can be reformulated as an optimal transport problem with multi-marginal constraints, as follows. First define the cost
\begin{equation}\label{cmin}
c(x):=\min_{z\in Z} \sum_{i=1}^I c_i(x_i,z),
\end{equation}
where $x=(x_1,\ldots, x_I)$.
Let $T(x) \in Z$ be a measurable selection of the solution of the above minimization, meaning that $T(x)\in Z$ satisfies
\[
\sum_{i=1}^I c_i(x_i,T(x))=c(x).
\]
Then consider the multi-marginal problem
\begin{equation}\label{mmk}
\inf_{\eta \in \Pi(\mu_1, \ldots, \mu_I)} \int_{X_1\times \ldots \times X_I} c(x) \eta(dx),
\end{equation}
where $\Pi(\mu_1, \ldots, \mu_I)$ denotes the set of probability measures on  $X_1\times \ldots \times X_I$ having $(\mu_1, \ldots, \mu_I)$ as marginals.  
It is not difficult to see that if $\eta$ solves \pref{mmk} then $\nu:=T_{\#} \eta$ solves \pref{primalmft} (where as usual $T_\# \eta$ denotes the push forward of $\eta$ through $T$, i.e. $T_\# \eta(B):=\eta(T^{-1}(B))$ for every Borel set $B$).

Conversely, one can relate the minimizers of \pref{primalmft} to those of \pref{mmk}. Indeed, let $\nu$ solve \pref{primalmft} and let $\gamma_i\in \Pi(\mu_i, \nu)$ be an optimal plan for $W_{c_i}(\mu_i, \nu)$. Disintegrating $\gamma_i$ with respect to $\nu$ i.e. writing $\gamma_i=\gamma_i^z \otimes \nu$ and defining $\gamma\in \PP(X_1 \times \ldots \times X_I \times Z)$ by:
\[\gamma:=\otimes_{i=1}^I \gamma_i^z \otimes \nu\] 
and $\eta$ as the marginal of $\gamma$ on the  variables $(x_1,\ldots, x_I)$, one easily checks that

\begin{itemize}

\item $\eta\in \Pi(\mu_1, \ldots, \mu_I)$ solves \pref{mmk},

\item on the support of $\gamma$, $\spt(\gamma)$, one has  
\[\sum_{i=1}^I c_i(x_i, z)=c(x),\]

\item the previous relation, together with the fact that $\nu$ is the $Z$-marginal of $\gamma$ and $\mu_i$ its $X_i$-marginal  then imply a useful  localization property: the support of the barycenter measure, $\spt(\nu)$, is contained in the set of minimizers of  the following problem.
\bq\label{localiz}
\begin{aligned}
\min_{z\in Z}   \sum_{i=1}^I c_i(x_i, z)   \text{ for some }   x_i \in \spt{\mu_i},\quad i = 1,\dots, I.  
\end{aligned}
\eq

\end{itemize}

Since the support of $\nu$ is unknown, which causes difficulties in practice, the localization property \eqref{localiz} gives a practical method for bounding the unknown support of the barycenter measure. The condition above results in a reduction of  the dimensionality of discretized problems since it gives an a priori information on the support of the unknown measure, at the expense of solving an optimization problem.  However this optimization problem is decoupled on the domain $Z$: each point (or neighborhood) can be tested by looping through points (or small neighborhoods) in the domain $Z$ and choices of points in the support sets $\spt(\mu_i)$.

In the case where $X_i$ and $Z$ coincide with some ball of  $\R^d$, and 
the costs are powers of distance,   $c_i(x_i, z)=\lambda_i \vert x_i -z \vert^p$ (with $\lambda_i>0$ and $\sum \lambda_i=1$, say) for some $p\ge1$, one can easily derive an information on the unknown support. Indeed, using the optimality condition for the minimization problem \eqref{localiz}, one deduces that $\spt(\nu)$ is included in  the convex hull of the supports of the $\mu_i$'s.  If we particularize further to  the Wasserstein barycenter case, i.e. to the case $p=2$, the solution of \eqref{localiz} is explictly given by the barycenter $z=\sum_{i=1}^I \lambda_i x_i$ so that the localization property \pref{localiz} gives the following estimate on the barycenter measure $\nu$:
\begin{equation}\label{localizbar}
\spt(\nu)\subset \sum_{i=1}^I \lambda_i \spt(\mu_i). 
 \end{equation}

\subsection{Linear programming formulation}\label{lpr}

Multi-marginals optimal transport problems such as \pref{mmk} are linear programs.  For discrete marginals, such problems can in principle be solved exactly by the simplex method. In practice however, the number of variables explodes with the number of marginals, which makes the problem quickly intractable. We shall see below that one may take advantage of the fact that $c$ is not any cost function but has the special structure \pref{cmin}. Interestingly, it was already proved by Pass \cite{pass1} in the context of multi-marginal optimal transport that such costs are much more well-behaved than arbitrary costs of $I$ variables.

To find  a more tractable linear programming reformulation of the matching for teams problem, it is better to go back to the very definition of an equilibrium in terms of couplings and to reformulate problem \pref{primalmft} as

\begin{equation}\label{lpcouplings}
\inf_{(\gamma_1, \ldots , \gamma_I)\in \Pi} \sum_{i=1}^I \int_{X_i\times Z}  c_i(x_i, z )\gamma_i(dx_i, dz)
\end{equation}
where $\Pi$ consists of all measures $(\gamma_1, \ldots , \gamma_I)\in \PP(X_1\times Z)\times \ldots \times  \PP(X_I\times Z)$ such that

\begin{itemize}

\item the marginal of $\gamma_i$ on the $x_i$ variable is $\mu_i$ i.e. 
\begin{equation}\label{contrmu}
\int_{X_i\times Z} \psi(x_i) \gamma(dx_i, dz)=\int_{X_i} \psi(x_i) \mu_i(dx_i), \; \forall \psi \in C(X_i), 
\end{equation}

\item the marginal of $\gamma_i$ on the $z$ variable does not depend on $i$:  
\begin{equation}\label{contrmu2}
\int_{Z} \varphi(z) \gamma_1(dx_1, dz)=\ldots=\int_{X_I} \varphi(z) \gamma_I(dx_I, dz), \; \forall \varphi \in C(Z).
\end{equation}

\end{itemize}

Clearly, if the $\gamma_i$'s solve \pref{lpcouplings} then their common marginal $\nu\in \PP(Z)$ solves \pref{primalmft} and the $\gamma_i$'s are optimal for the optimal transport problem $W_{c_i}(\mu_i, \nu)$. In other words, the $\gamma_i$'s are equilibrium couplings for the matching for teams problem.

\smallskip

The constraints above being linear, $\Pi$ is a convex and weakly $*$ compact subset of $\PP(X_1\times Z)\times \ldots \times  \PP(X_I\times Z)$ so that \pref{lpcouplings} admits solutions. Moreover in the case of discrete $\mu_i$'s  and $\nu$ supported by $N$ points, the number of variables in the linear program \pref{lpcouplings} is  linear (and not exponential as in the case of the multi-marginal optimal transport problem) in the number of marginals.

\subsection{Discretization}\label{sec:discretization}
The (a priori) infinite dimensional linear programming problem \eqref{lpcouplings} of \autoref{lpr} can be discretized as follows. 
Let $\{ S_i^j \}_{j = 1}^{N_i}$ be a partition of $\spt(\mu_i)$ and let $\{ S^k_0 \}_{k = 1}^{N_0}$ be a partition of $Z$ (or better, of the support set estimated by the method of \autoref{sec:localization}).  Approximate the measures by weighted sums of atoms 
\begin{align*}
\mu_i^A &= \sum_{j=1}^{N_i} \mu_i^j \delta_{x_i^j}, \quad  \text{ for } i = 1,\dots, I, \; \text{ with } \mu_i^j=\mu_i(S_i^j)
\\
\nu^A &= \sum_{k=1}^{N_0} \nu^k \delta_{z^k},  \; \text{ with } \nu^k=\nu(S^k_0) &
\end{align*}
where  ${x^j_i}$ and ${z^k}$ are representative points in the regions ${S^j_i}, {S^k_0}$, respectively. It is well-known, that $\mu_i^A$ converges weakly $*$ to $\mu_i$ as the diameter of the partition  $\{ S_i^j \}_{j = 1}^{N_i}$  tends to $0$. More precisely, denoting by $W_1$ the $1$-Wasserstein distance (which metrizes the weak $*$ topology on probability measures):
\begin{equation}\label{approxdiscr}
W_1(\mu_i^A, \mu_i) \le \max_{j=1,\dots, N_i}  \diam(S_i^j).
\end{equation}

Inserting the approximation defined above into the linear programming problem ~(\ref{lpcouplings}, \ref{contrmu}, \ref{contrmu2})
results in the following finite dimensional linear programming problem 
\bq\label{DLP}
\begin{aligned}
\text{minimize } &\quad \sum_{i=1}^I  \sum_{j,k} c_i (x^j_i, z^k)  \gamma^{j,k}_i &&   \\
\text{subject to:   }  \\
&\sum_{k} \gamma^{j,k}_i = \mu_i^j,  &&  \text{ for all }  i =1,\dots, I , \text{ and  } j=1,\dots,N_i
\\
&\sum_{j} \gamma^{j,k}_1 = \ldots =   \sum_{j} \gamma^{j,k}_I, &&  \text{ for all }\quad k=1,\dots,N_0,&&
\end{aligned}
\eq
along with the non-negativity constraints $\gamma^{j,k}_i \ge 0.$   The (approximated) barycenter is then $\nu^A = \sum_k \nu^k \delta_{z^k}$ where the weight $\nu^k$ is  given by 
the common value,
\[
\nu^{k}  = \sum_{j} \gamma^{j,k}_i, \quad \text{ for any } i= 1,\dots,I. 
\]
The linear programming problem above can be implemented in standard software packages.  The size of the problem above is as follows.
 The number of variables is $N_0\times (N_1 + \dots + N_I)$  
(or $I N^2 $ if each $N_0 = N_1 = \dots = N_I = N$).   The number of constraints is $(N_0+1)\times (N_1 + \dots + N_I)+I N_0$ (that is $I(N^2+2N)$ when $N_0 = N_1 = \dots = N_I = N$). The size of this linear programming problem  thus scales linearly with the number of marginals, for a given, fixed value of $N$ (contrary to the multi-marginal formulation \eqref{mmk}).

\subsection{Approximation and convergence}\label{sec:convergence}

Since in practice, one considers approximation by discrete measures just as in \autoref{sec:discretization}, we wish now to address the stability of  the  following convex problem when one replaces the measures $\mu_i$ by some discrete approximation
\begin{equation}\label{primalmft2}
\inf_{\nu\in \PP(Z)} J(\nu):= \sum_{i=1}^I W_{c_i} (\mu_i, \nu).
\end{equation}
To do so, one has to control the dependence of $W_c(\mu, \nu)$ in its three arguments $(c,\mu, \nu)\in C(X\times Z)\times \PP(X)\times \PP(Z)$. We shall denote by $d_X$ and $d_Z$ the distances on $X$ and $Z$, take now $(c,\mu, \nu)\in C(X\times Z)\times \PP(X)\times \PP(Z)$ and $(\tc, \tmu, \tnu)\in C(X\times Z)\times \PP(X)\times \PP(Z)$ and let $\omega_X$ and $\omega_Z$ be respectively a modulus of continuity of $c$ and $\tc$ with respect to $x$ uniform in $z$ and a modulus of continuity of $c$ and $\tc$ with respect to $z$ uniform in $x$, that is
\[ \max(    \vert c(x,z)-c(x',z)  \vert, \vert \tc(x,z)-\tc(x',z)  \vert)  \le \omega_X (d_X(x, x')), \forall (x,x',z)\in X\times X\times Z \]
and
\[ \max(    \vert c(x,z)-c(x,z')  \vert, \vert \tc(x,z)-\tc(x,z')  \vert)  \le \omega_Z (d_Z(z, z')), \forall (x,z,z')\in X\times Z\times Z.\]

Obviously, one has
\begin{equation}
\vert W_c(\mu, \nu)-W_{\tc}(\mu, \nu) \vert \le \Vert c- \tc \Vert_\infty.   
\end{equation}\label{movc}
Let $\varphi \in C(Z)$ be  a solution in the Kantorovich dual of $W_{\tc}(\mu, \nu)$, that is
\[W_{\tc}(\mu, \nu)=\int_X \varphi^{\tc} d\mu +\int_Z \varphi d\nu\]
by the Kantorovich duality formula, we have
\[W_{\tc}(\tmu, \nu)\ge \int_X \varphi^{\tc} d\tmu +\int_Z \varphi d\nu.\]
Hence, for every $\theta\in \Pi(\mu, \tmu)$, we have 
\[W_{\tc}(\mu, \nu)-W_{\tc}(\tmu, \nu) \le \int_X \varphi^{\tc} d(\mu-\tmu)=\int_{X\times X} (\varphi^{\tc}(x)-\varphi^{\tc}(x')) \theta(dx, dx').\] 
We then observe that $\varphi^{\tc}(x)-\varphi^{\tc}(x')\le \omega_X (d_X(x,x'))$ so that
\begin{equation}\label{movmu}
W_{\tc}(\mu, \nu)-W_{\tc}(\tmu, \nu) \le W_{\omega_X} (\mu, \tmu):=\inf_{\theta\in \Pi(\mu, \tmu)} \int_{X\times X} \omega_X(d_X(x,x'))  \theta(dx, dx'). 
\end{equation}
Similarly
\begin{equation}\label{movnu}
W_{\tc}(\tmu, \nu)-W_{\tc}(\tmu, \tnu) \le W_{\omega_Z} (\nu, \tnu):=\inf_{\eta\in \Pi(\nu, \tnu)} \int_{Z\times Z} \omega_Z(d_Z(z,z'))  \eta(dz, dz'). 
\end{equation}
Putting everything together, we get
\begin{equation}\label{moveverything} 
\vert W_c(\mu, \nu)-W_{\tc}(\tmu, \tnu)\vert \le \Vert c- \tc \Vert_\infty+  W_{\omega_X} (\mu, \tmu)+ W_{\omega_Z} (\nu, \tnu). 
\end{equation}

We then remark that if $\mu_n$ weakly $*$ converges to $\mu$ then $W_{\omega_X} (\mu, \mu_n)\to 0$. Indeed, it is known to imply that the $1$-Wasserstein distance (corresponding to $W_{\omega_X}$ for $\omega_X(t)=t$) between $\mu_n$ and $\mu$ converges to $0$, so that there is some $\theta_n \in \Pi(\mu, \mu_n)$ which (up to a non relabeled subsequence) weakly $*$ converges to some $\theta$ supported on the diagonal of $X\times X$, hence
\[ W_{\omega_X} (\mu, \mu_n) \leq \int_{X\times X} \omega_X(d(x,x')) \theta_n(dx,d x') \to 0.\]

Getting back to the approximation of \pref{primalmft}, take sequences $c_i^n\in C(X_i\times Z)$, $\mu_i^n\in \PP(X_i)$, and $c_i\in C(X_i\times Z)$, $\mu_i \in \PP(X_i)$, such that
\begin{equation}
\Vert c_i^n -c_i\Vert_{\infty} \to 0, \; \mu_i^n  \weakstarto \mu_i
\end{equation}  
and set:
\begin{equation}\label{defJJn}
J(\nu):=\sum_{i=1}^I W_{c_i} (\mu_i, \nu), \; J_n(\nu):= \sum_{i=1}^I W_{c_i^n} (\mu_i^n, \nu),\;   \forall \nu\in \PP(Z).
\end{equation}
Denoting by $\omega^i_{X_i}$ and $\omega^i_Z$ common continuity modulus of the $c_i^n$ (the first one in $x_i$ uniformly in $z$ and the second in $z$, uniformly in $x_i$ just as above)  we then have:

\begin{prop}
For every $(\nu, \nu_n)\in \PP(Z)\times \PP(Z)$ :
\begin{equation}\label{cvjnj}
\vert J(\nu)-J_n(\nu_n) \vert \le \sum_{i=1}^I \Vert c_i -c_i^n \Vert_{\infty} + \sum_{i=1}^I [W_{\omega^i_{X_i}} (\mu_i^n, \mu_i)+W_{\omega^i_{Z}} (\nu_n, \nu)]
\end{equation}
this implies in particular 

\begin{itemize}

\item $J_n(\nu_n)\to J(\nu)$ whenever $\nu_n \weakstarto \nu$,

\item a quantitative estimate for the stability of values:

\begin{equation}\label{conval}
\vert \inf_{\PP(Z)} J- \inf_{\PP(Z)} J_n \vert \le  \sum_{i=1}^I \Vert c_i -c_i^n \Vert_{\infty} + \sum_{i=1}^I W_{\omega^i_{X_i}} (\mu_i^n, \mu_i)
\end{equation}

\item if $\nu_n$ minimizes $J_n$ then, up to a subsequence, it weakly $*$ converges to a minimizer of $J$.

\end{itemize}

\end{prop}

\begin{proof}

The statements directly follow from estimate \pref{moveverything} and the already observed fact that the right-hand side of \pref{cvjnj} converges to $0$ as soon as $\nu_n \weakstarto \nu$. \end{proof}

In the case where the cost functions $c_i=c_i^n$ are Lipschitz (with Lipschitz constant ${\mathrm{Lip}}(c_i)$) and the approximated measures $\mu_i^n$ satisfy (for the usual $1$-Wasserstein distance) $W_1(\mu_i^n, \mu_i) \le \frac{C}{n}$, \pref{cvjnj} above just takes the form
\[ \vert J(\nu)-J_n(\nu_n) \vert \le  \sum_{i=1}^I  {\mathrm{Lip}}(c_i) \Big( \frac{C}{n} +W_1 (\nu^n, \nu)\Big). \]

\section{Numerical simulations: Linear Programming \label{sec:num}}

In this section, we present various numerical simulations using the Linear Programming approach of sections \ref{lpr} and \ref{sec:discretization}.   The localization method of \ref{sec:localization} is used to approximate the support of the barycenter measure.   
An alternate approach to approximating the support of the barycenter measure, which can be combined with localization is a 
 a two stage solution approach:  the first stage, using a coarse grid, gives an approximation of the support of the barycenter, the second stage gives a more accurate representation of the barycenter using information on the support obtained in the first stage.

 All computations in this section were performed in MATLAB on a Mid 2011 MacBook Air laptop.  
To solve \eqref{DLP} we use the software package CVX \cite{cvx} \cite{gb08} which is callable from  MATLAB.   The CVX language allows for a very concise description of the convex optimization problem, and allows for the use of multiple solver libraries (e.g.\ MOSEK, Gurobi).  The numerical  solution obtained is the correct up to tolerances near numerical precision. 



\begin{figure}[htdp]
	\centering
	
\hspace{-2in}
	\includegraphics[width=.33\textwidth]{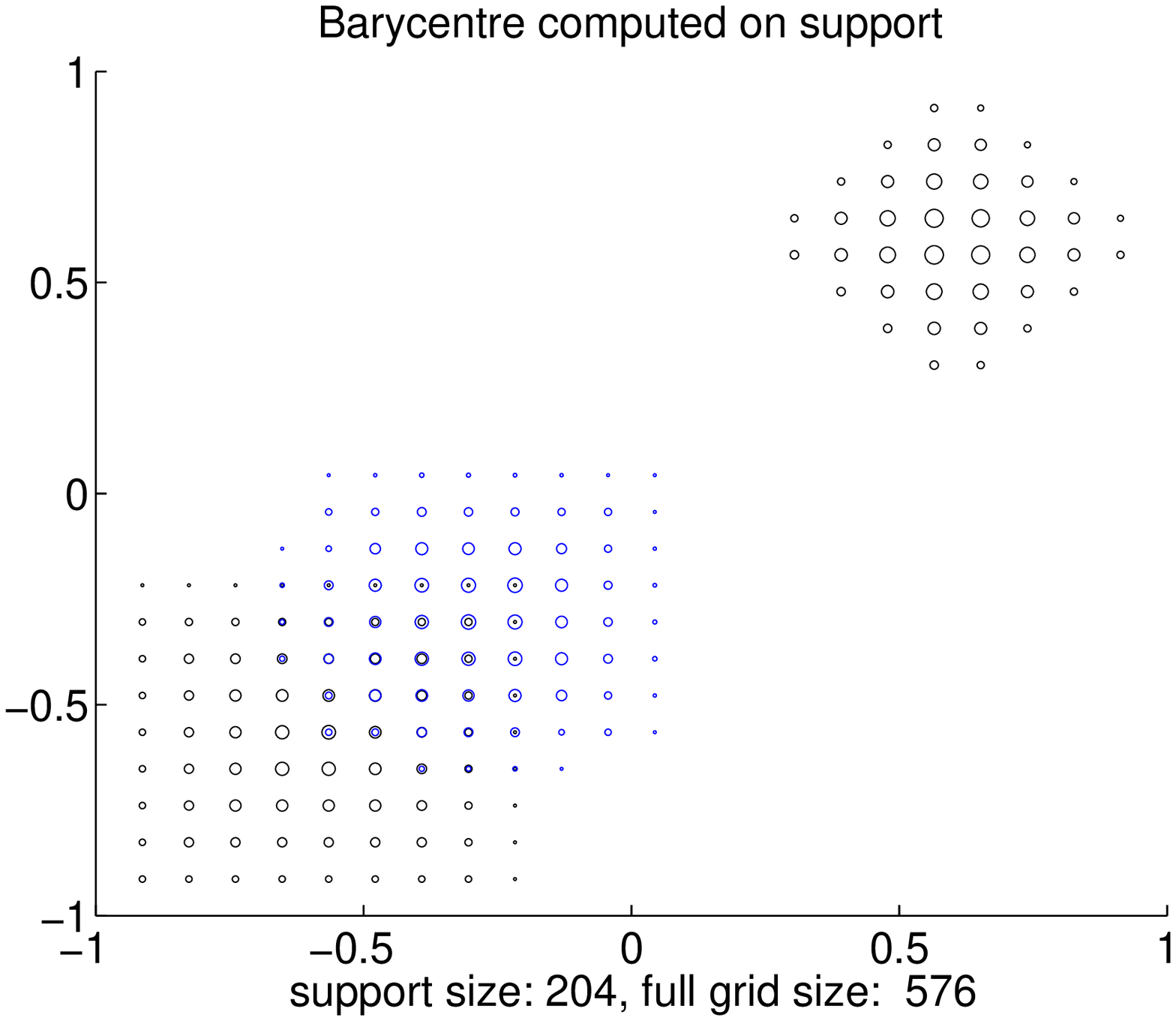}
	\includegraphics[width=.33\textwidth]{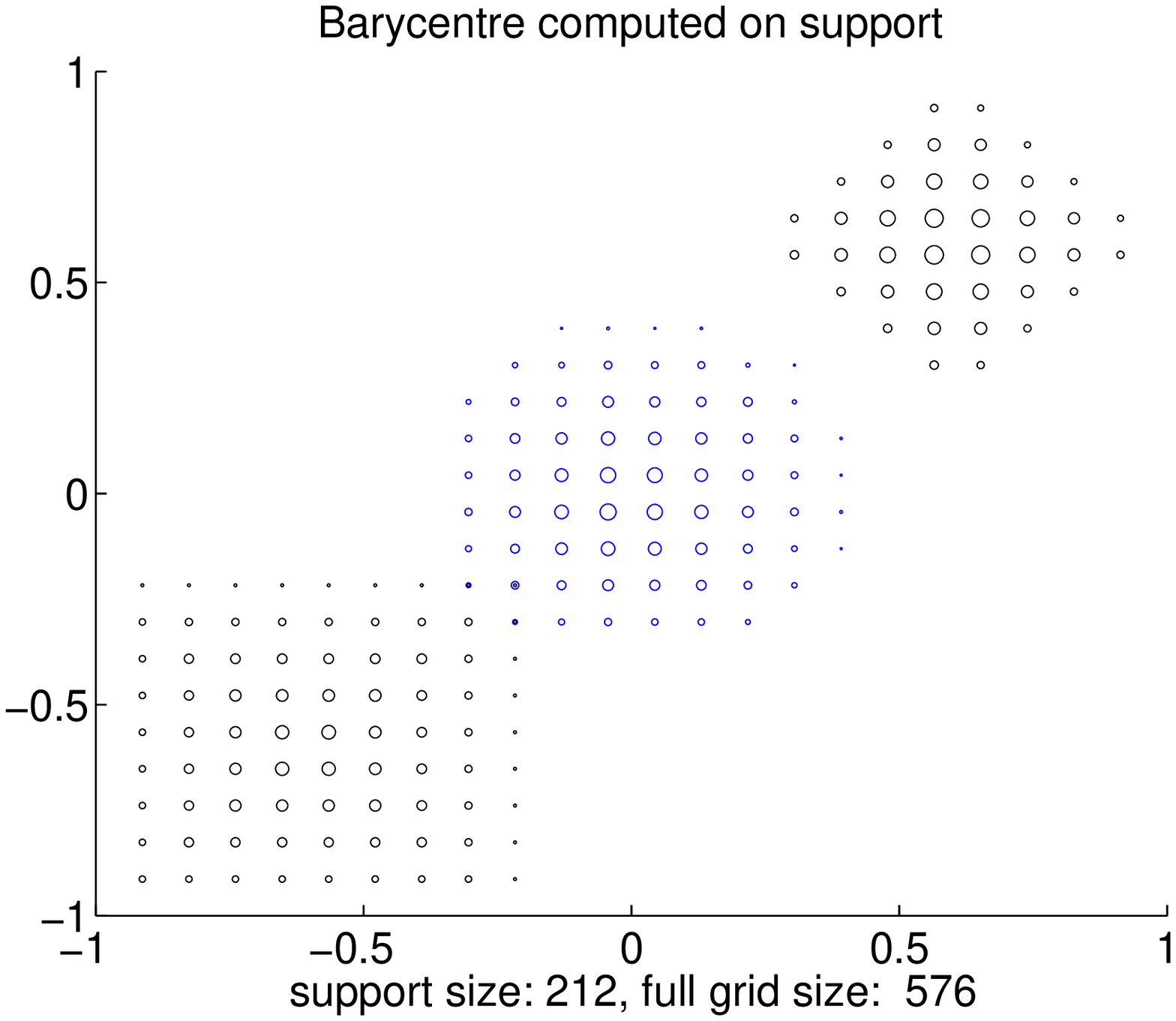}	
	\includegraphics[width=.33\textwidth]{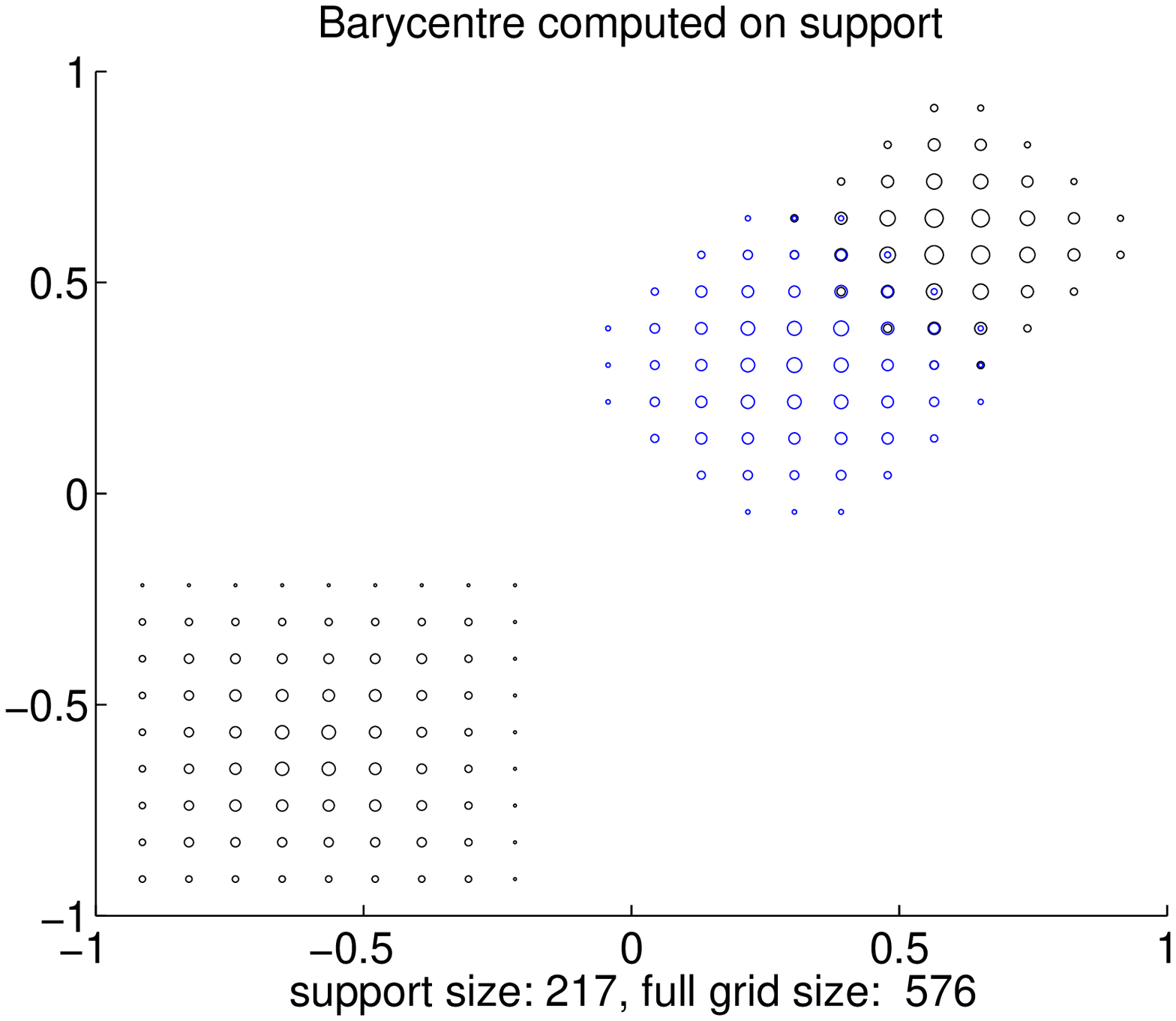}
\hspace{-2in}

\caption{Given measures in black.  Figure left, centre, right: solution of the geodesic problem with weights .25, .5, .75, respectively.}
\label{fig:geo}
\end{figure}

\subsection{Geodesic paths between measures in the plane
}
We considered two measures in the plane, and by varying the weights in the quadratic cost function, we computed three points on the  geodesic path (or McCann interpolant) represented in Figure~\ref{fig:geo}. The computational time was less than a minute.  The measures are illustrated by a circle centered on the atom (middle of the corresponding square) and a radius proportional to the weight.  
Both the shape of the support (square, diamond) and the density of the measure are illustrated in the figure: the interpolated measures are influenced by both properties. Figure~\ref{fig:twoDcoarseandrefined} illustrates the two-stage support refinement strategy.
\begin{figure}[htdp]
	\centering
	\includegraphics[width=.33\textwidth]{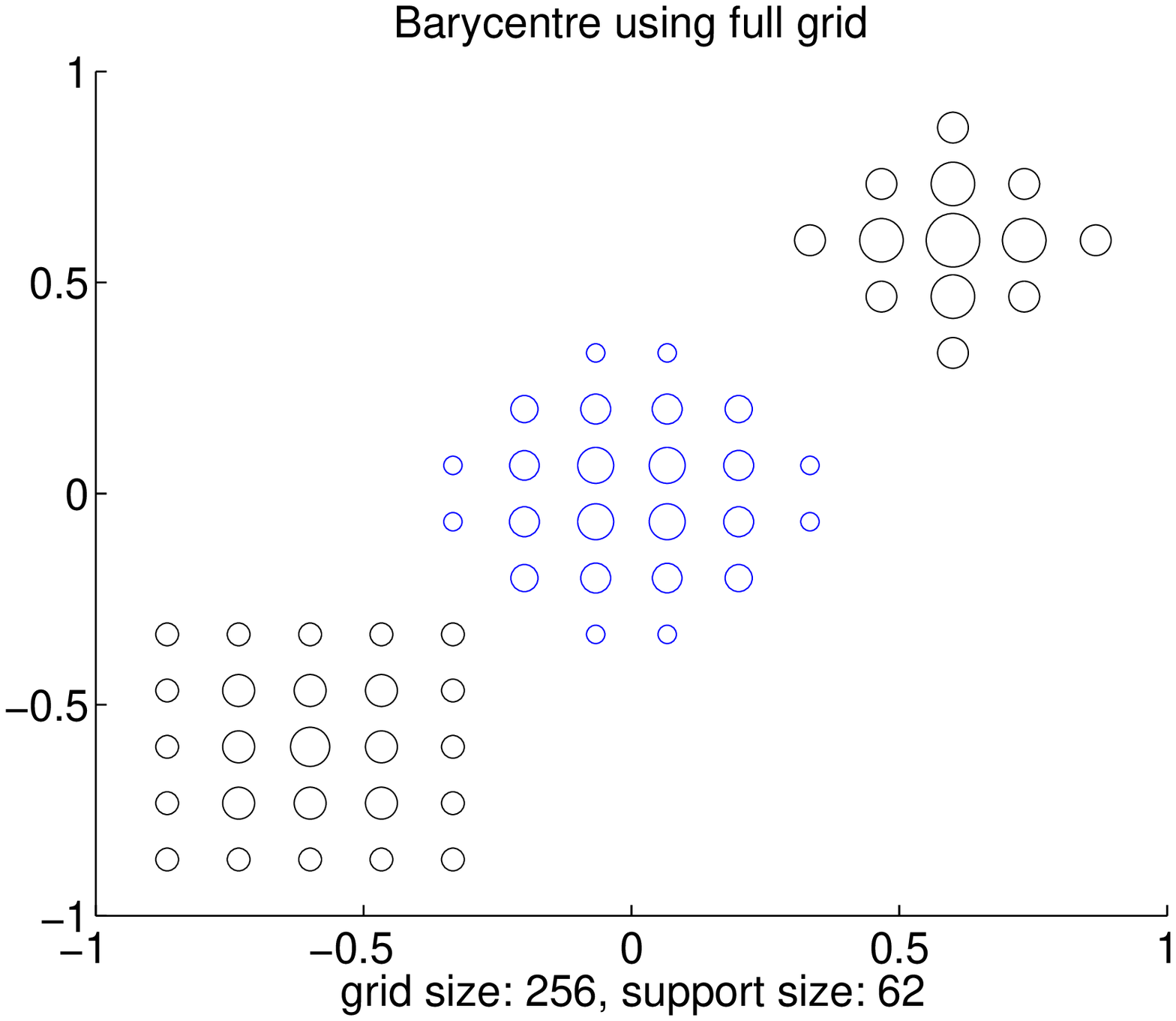}
	\includegraphics[width=.33\textwidth]{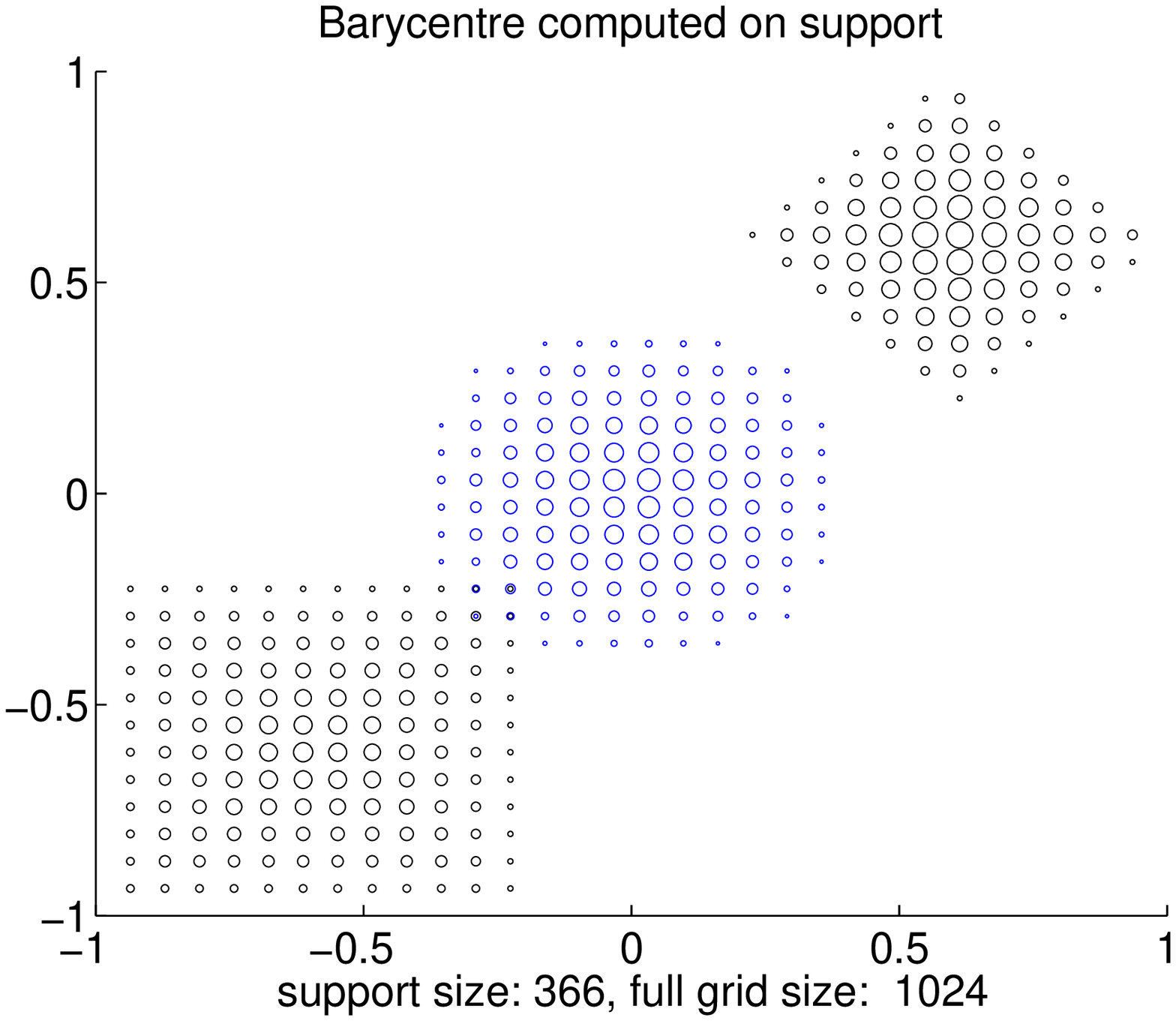}	

\caption{Refinement for the geodesic problem.  Left: solution on full grid.  Right: solution on the implied support, but more resolved. The size of the problems is the same, but the resolution is increased by a factor of three.  
}
\label{fig:twoDcoarseandrefined}
\end{figure}





\subsection{Comparing different cost functions}
For the next set of examples, we took two uniform measures, the first corresponding to a vertically oriented rectangle, and the second corresponding to an horizonal rectangle.  These measures  are shown in Figure~\ref{fig:measures}. First, we compared the convergence of the solutions for different grid sizes in Figure~\ref{fig:resolution}.  Notice that the general support of the computed measures seems stable, but there are oscillations in the density, for different resolutions.

Next, we computed the barycenter with various power cost functions $C(x,y) = |x-y|^p$, for $p = 1, 2, 3, 4$.  
The solutions we computed use grids of size $50^2$.  Computational time was close to two minutes.   A second run using grid size of $100^2$ and  a localization of the support took 30 to 45 minutes and is represented in Figure~\ref{fig:costs}.  The densities are plotted using a grayscale which corresponds to the relative values, however the grayscale is different for each figure.  
We also include another view of the density for $p=1$ in Figure~\ref{fig:surfaceplot}.

The solutions have a complicated geometry.   For the cost with $p=1$, the support of the barycenter is the entire convex hull of the supports of the measures, although the density is highly concentrated at the intersection of the measures.  The density ranges from about 0.01 at the edges to 0.12 in the center.

For the case $p=2$, the density is supported on a square, but wider than the width of the rectangle.  The density has some oscillations, but is strictly positive (taking values in the range [.004,.005]).    For the case $p=3$, the density is supported on a small octagonal shape, with zero density in the middle, and with larger oscillations.  For the case $p=4$, a much larger octagonal shape appears with a large zero density hole in the middle.   The supports of the barycenter measures are close to the ones estimated by localization.
%
%
%
%
%

\begin{figure}[htdp]
	\centering	
	\includegraphics[width=.18\textwidth]{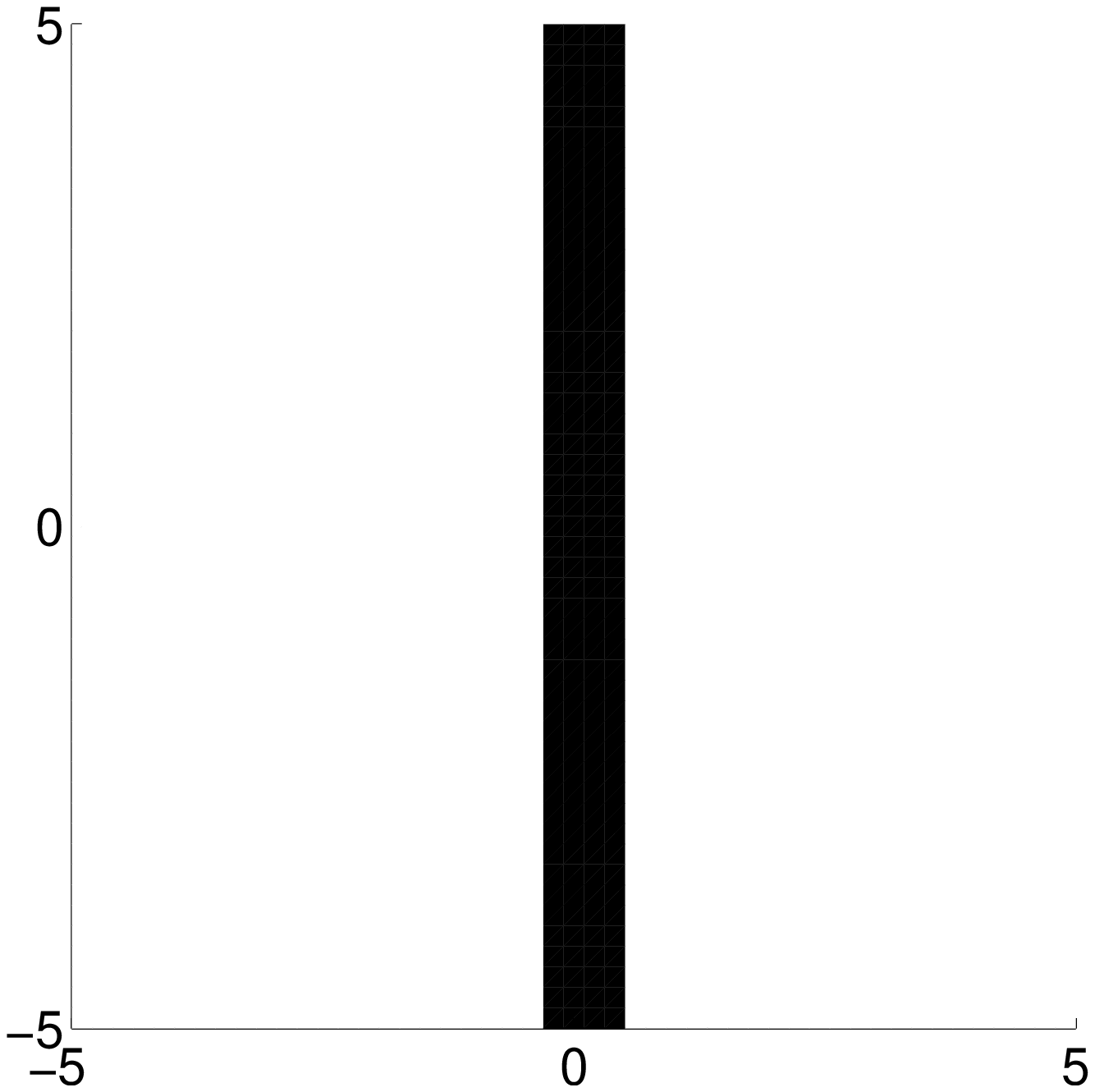}
	\includegraphics[width=.18\textwidth]{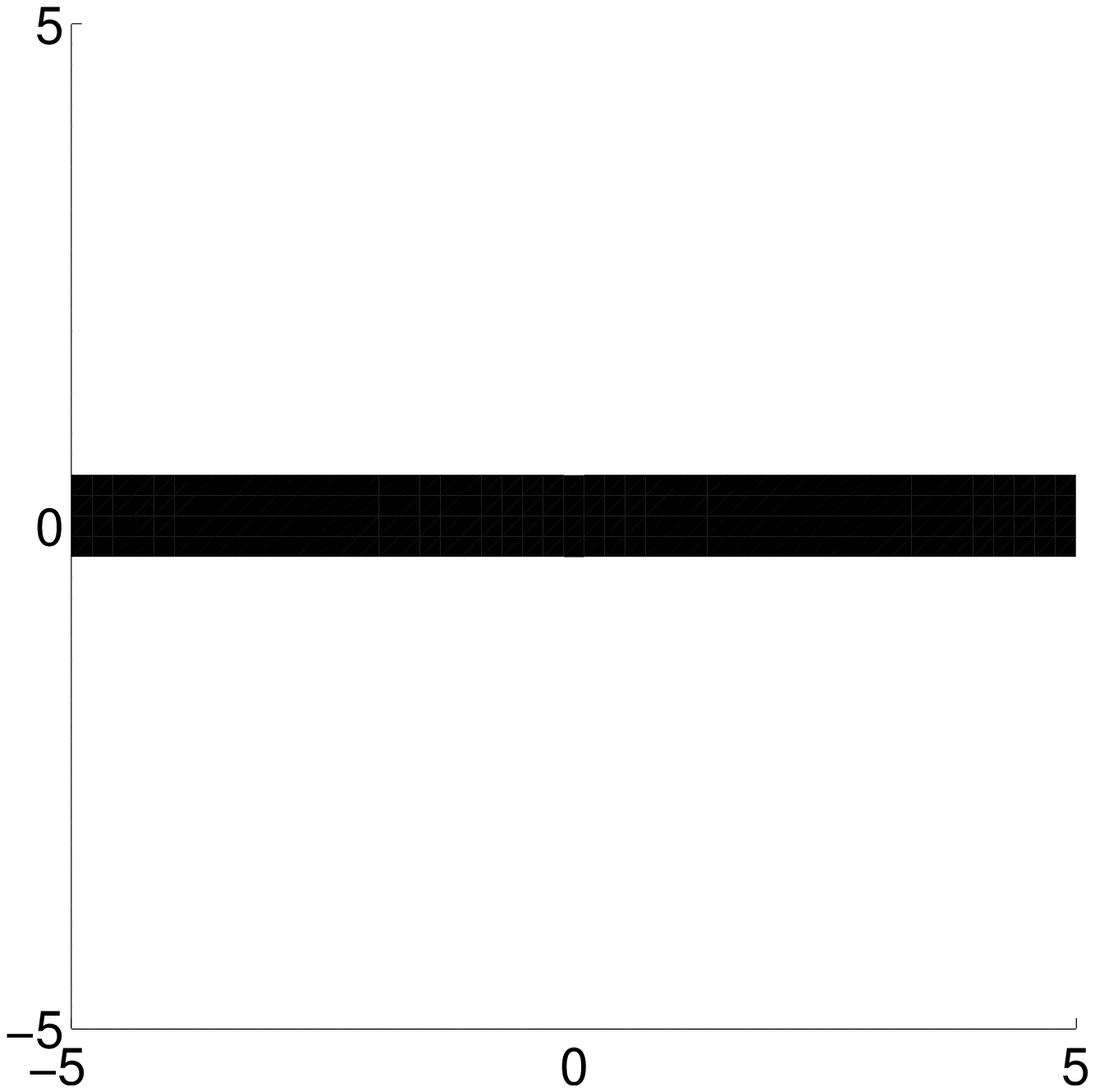}
		\includegraphics[width=.18\textwidth]{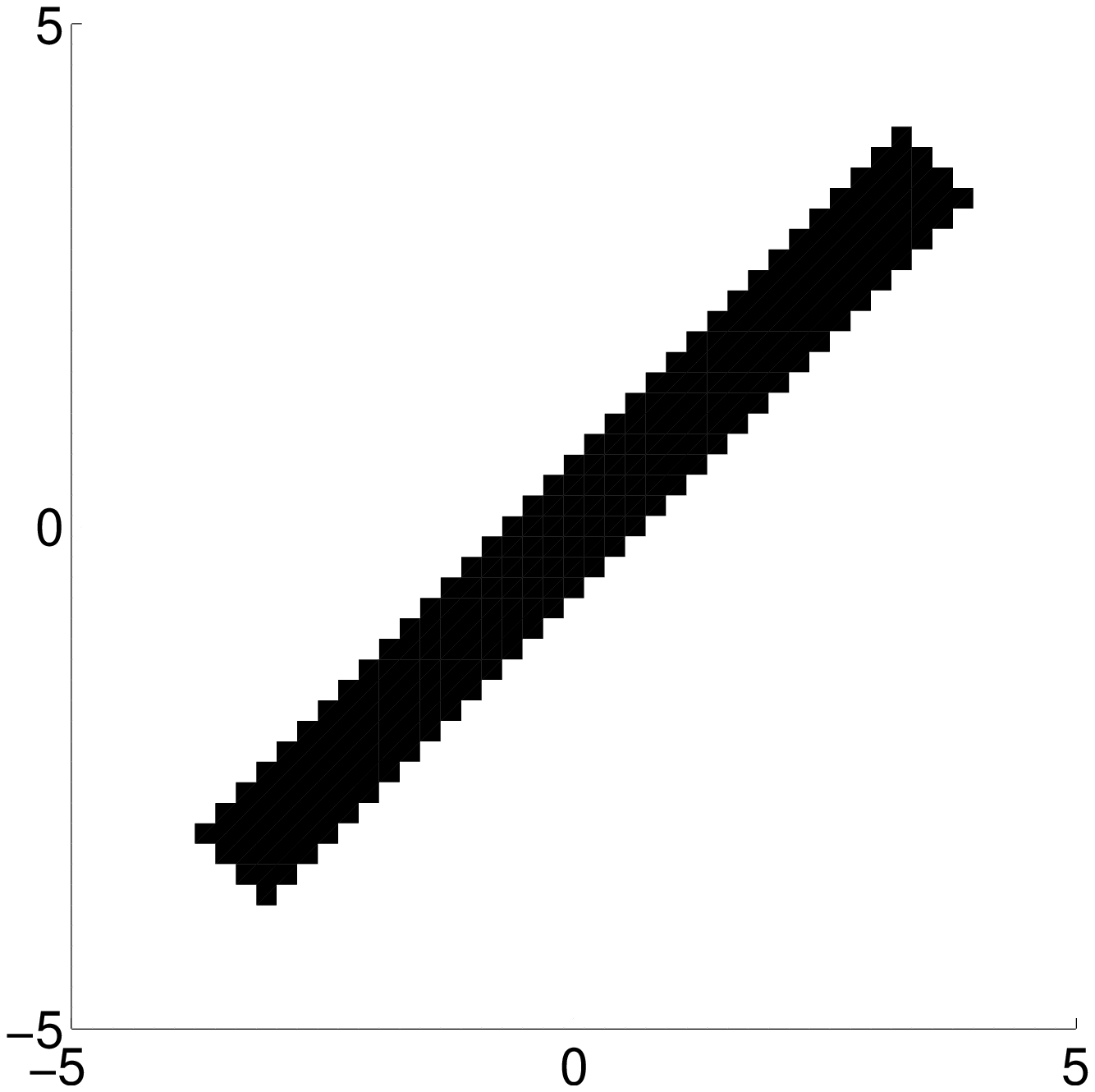}
	\includegraphics[width=.18\textwidth]{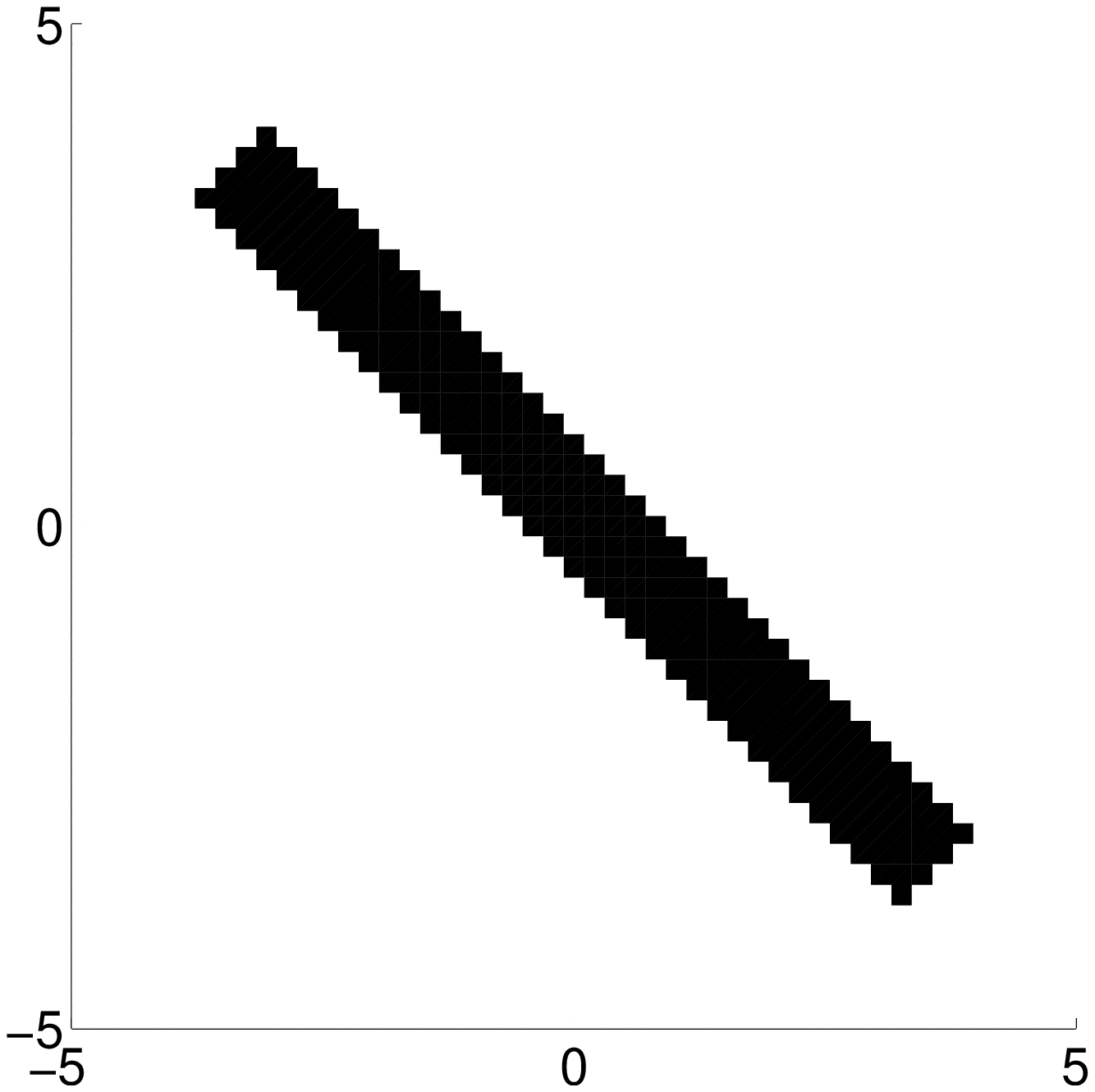}		
	\caption{The measures $m_1, m_2, m_3, m_4$ used in the examples which follow.}
\label{fig:measures}
\end{figure} 

\begin{figure}[htdp]
	\centering	
	\includegraphics[width=.24\textwidth]{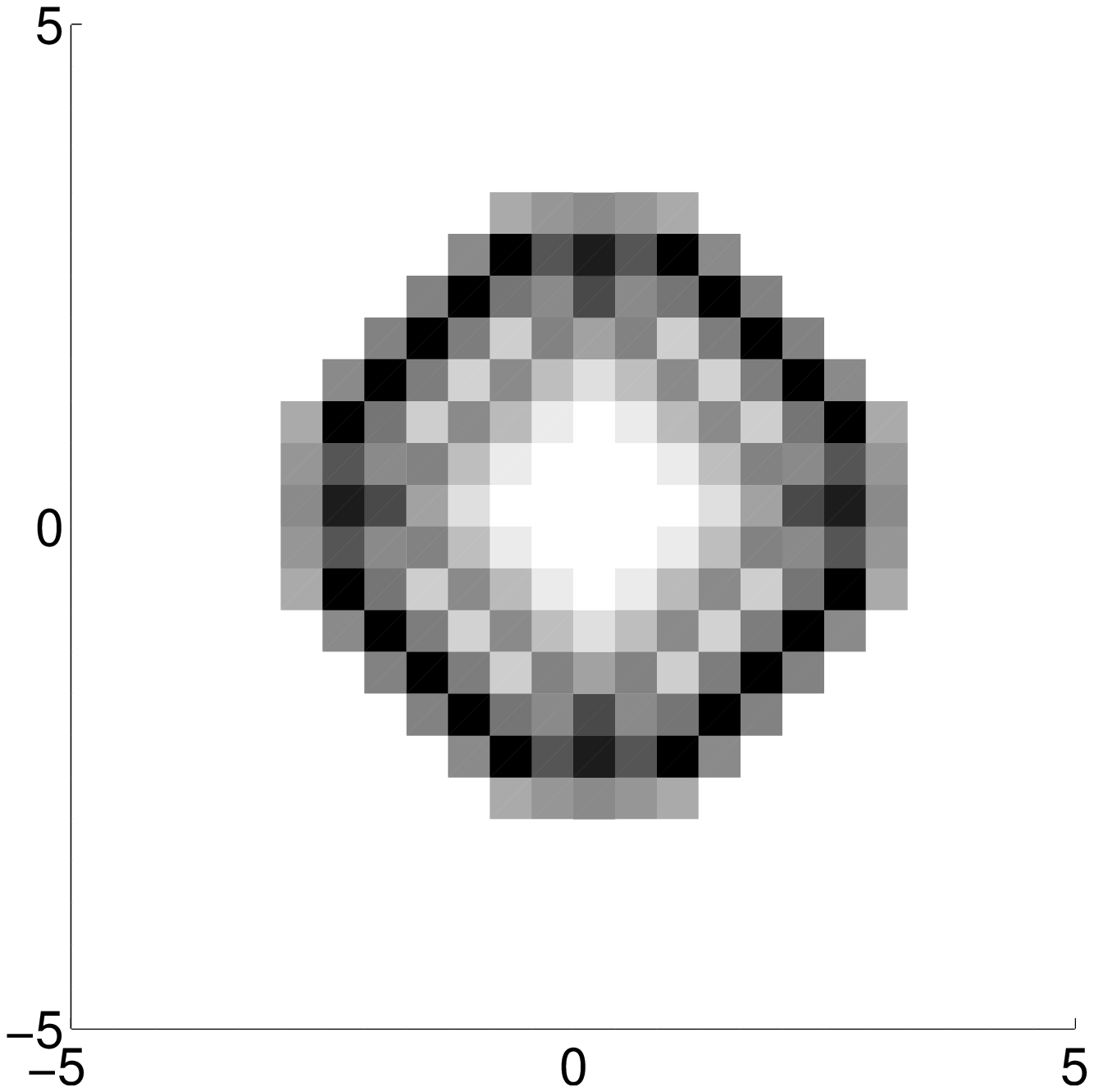}
	\includegraphics[width=.24\textwidth]{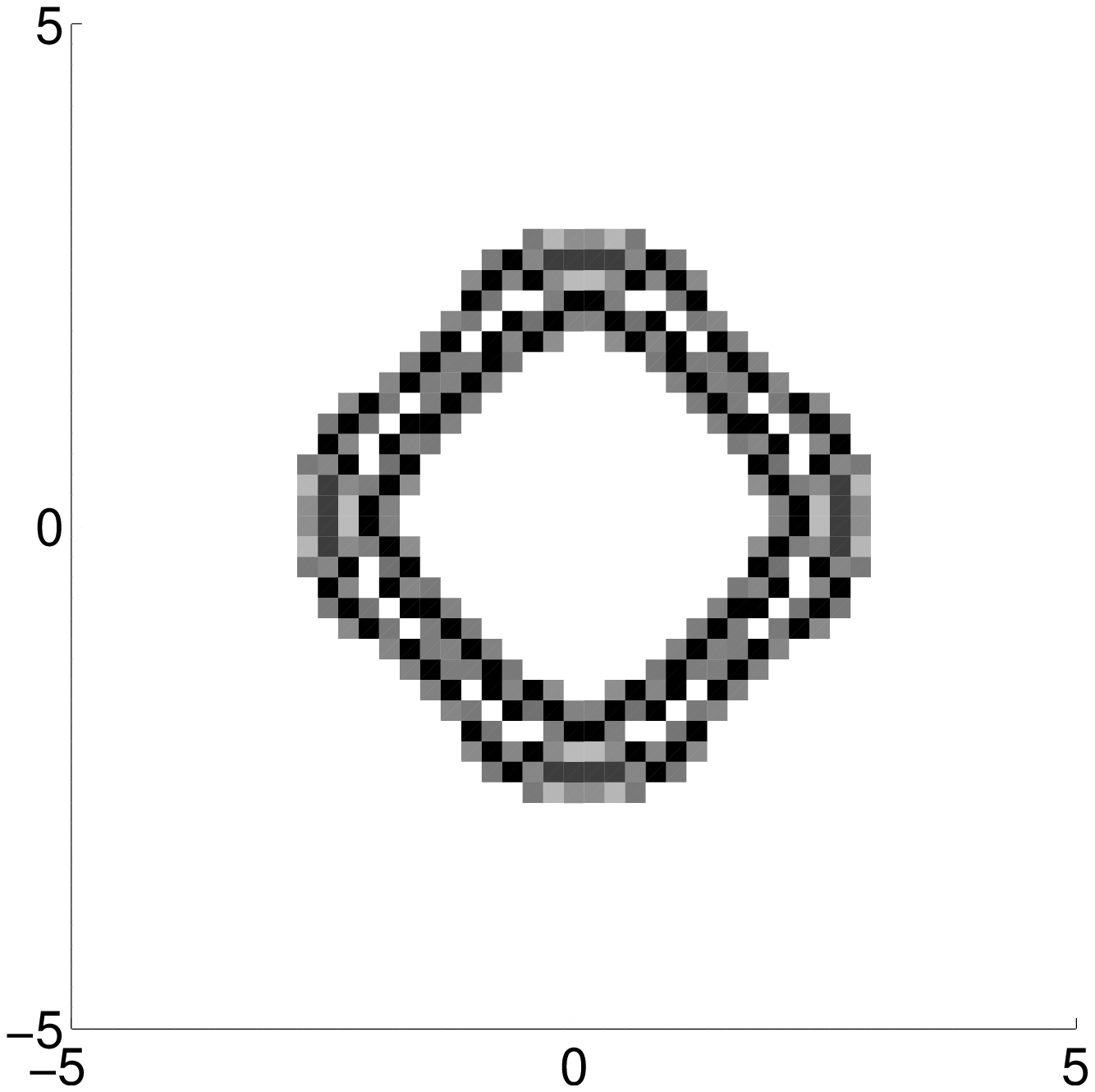}
		\includegraphics[width=.24\textwidth]{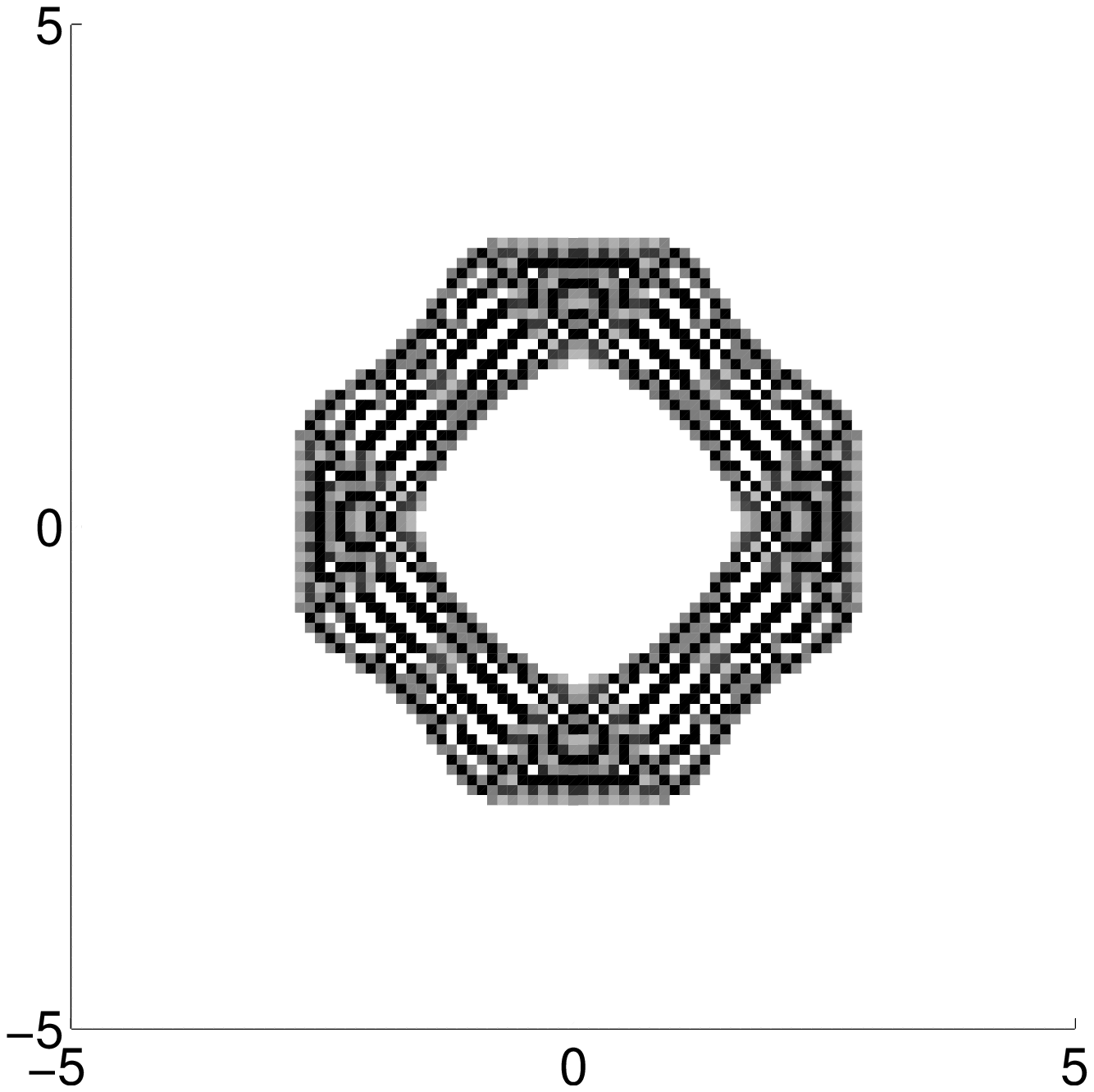}
	\caption{Comparison of the numerical barycenter for measures $m_1, m_2$ using cost $C(x,y) = |x-y|^4$ on different grid sizes: $25^2, 50^2$ and $100^2$.  Note the general shape of the solutions are similar, but the density has more oscillations at higher resolution.  }
\label{fig:resolution}
\end{figure}

\begin{figure}[htdp]
	\centering
	\includegraphics[width=.24\textwidth]{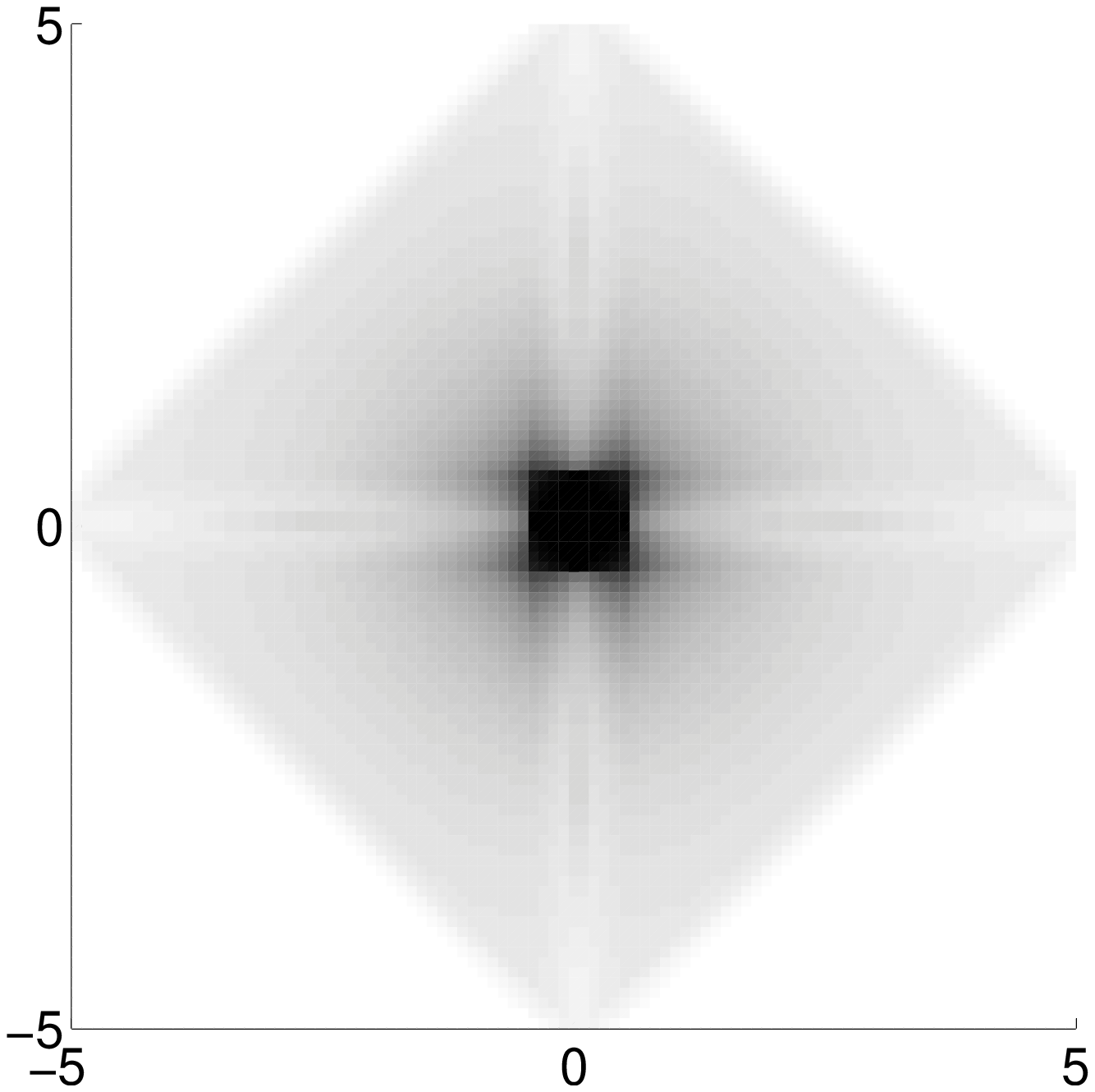}
	\includegraphics[width=.24\textwidth]{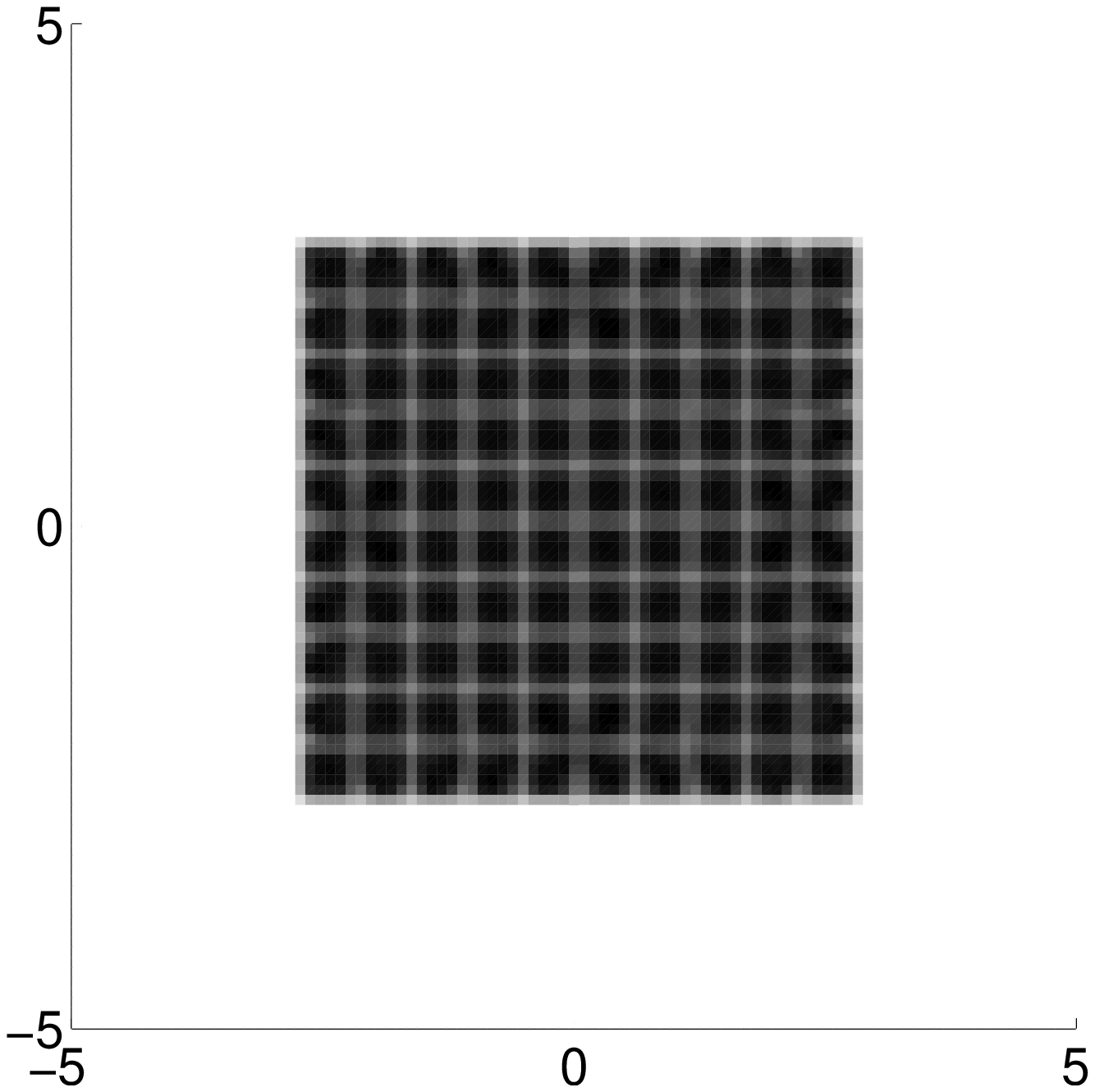}
	\includegraphics[width=.24\textwidth]{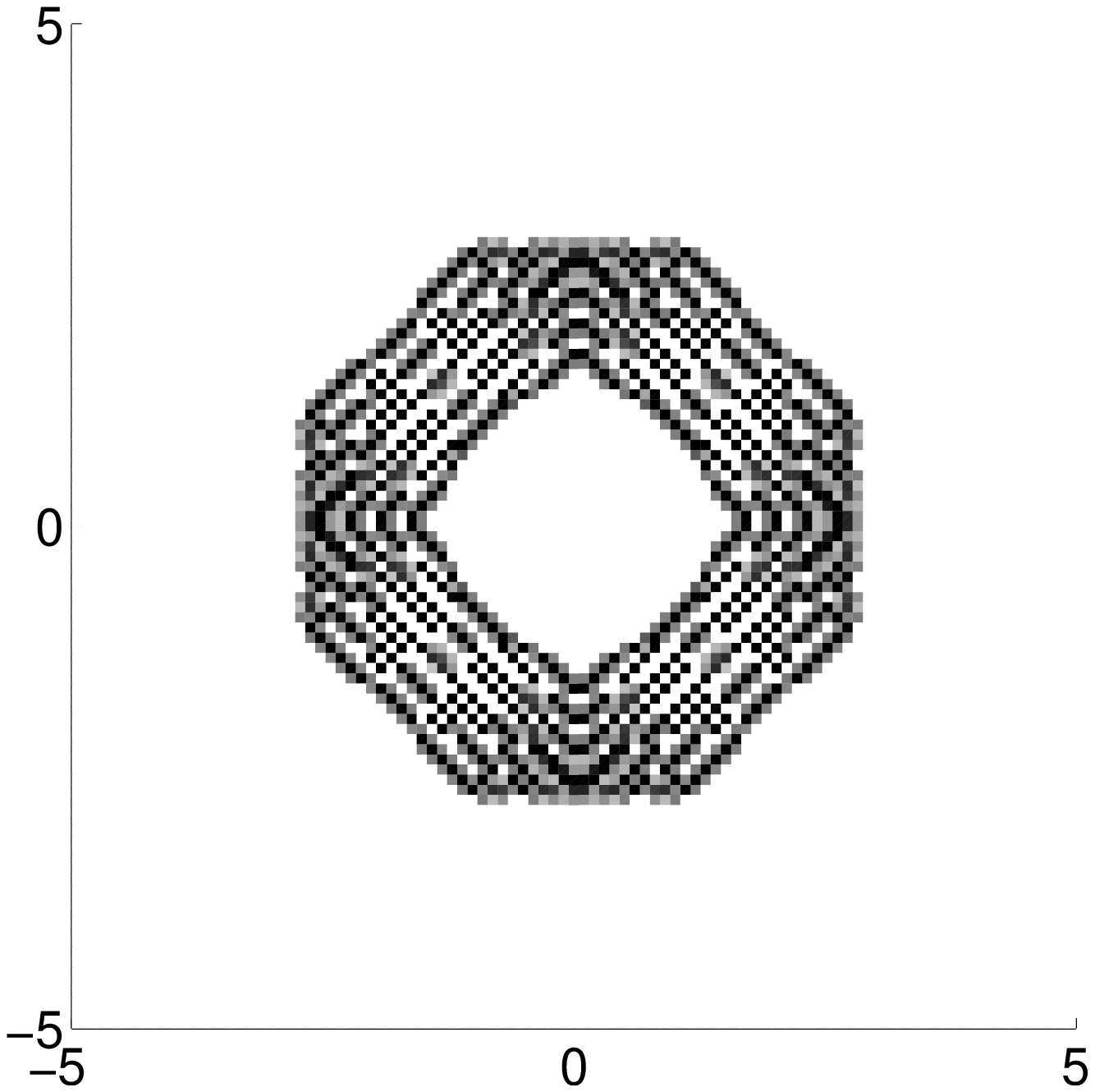}
	\caption{Barycenters of two rectangles $m_1, m_2$, with cost $C(x,y) = |x-y|^p$ for $p = 1, 2, 3$,
	using grid size $100^2$.	
	}
\label{fig:costs}
\end{figure}

\begin{figure}[htdp]
	\centering
		\includegraphics[width=.38\textwidth]{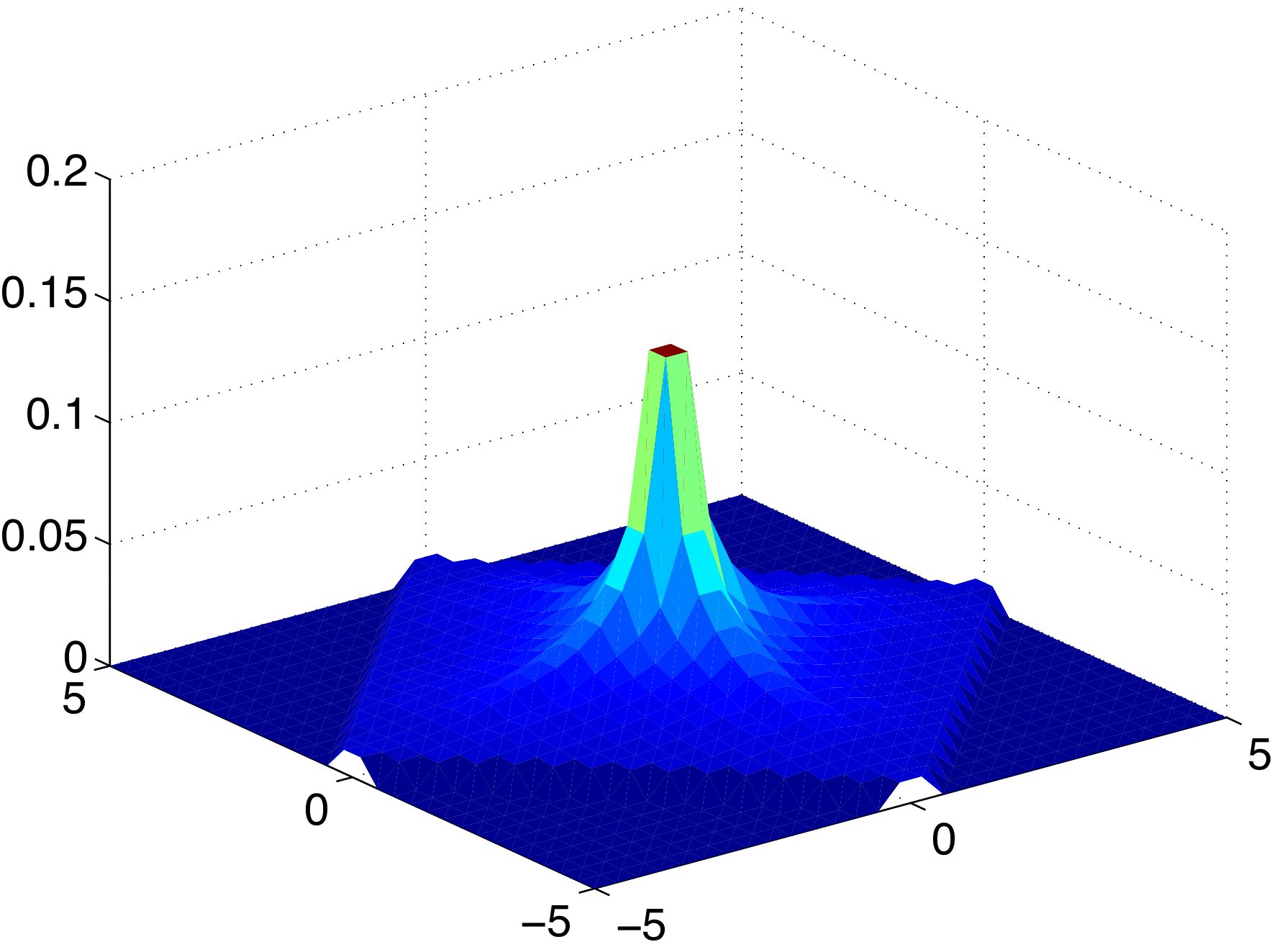}
	\caption{Surface plot of the barycenter corresponding to $p=1$.}
\label{fig:surfaceplot}
\end{figure} 

Finally, we computed the barycenters using all four measures from figure~\ref{fig:measures}, see figure~\ref{fig:4meas}. 

\begin{figure}[htdp]

	\includegraphics[width=.24\textwidth]{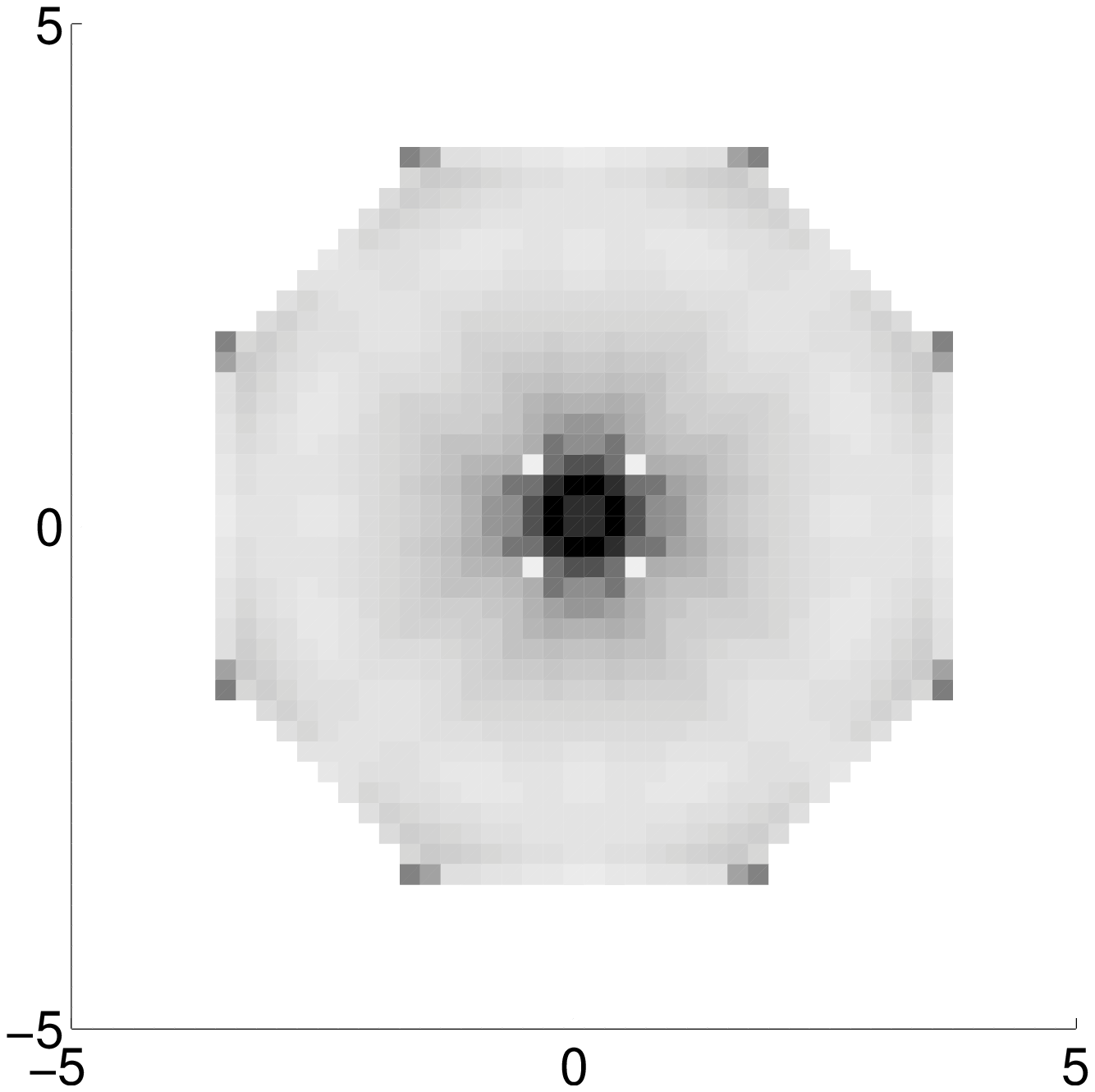}
	\includegraphics[width=.24\textwidth]{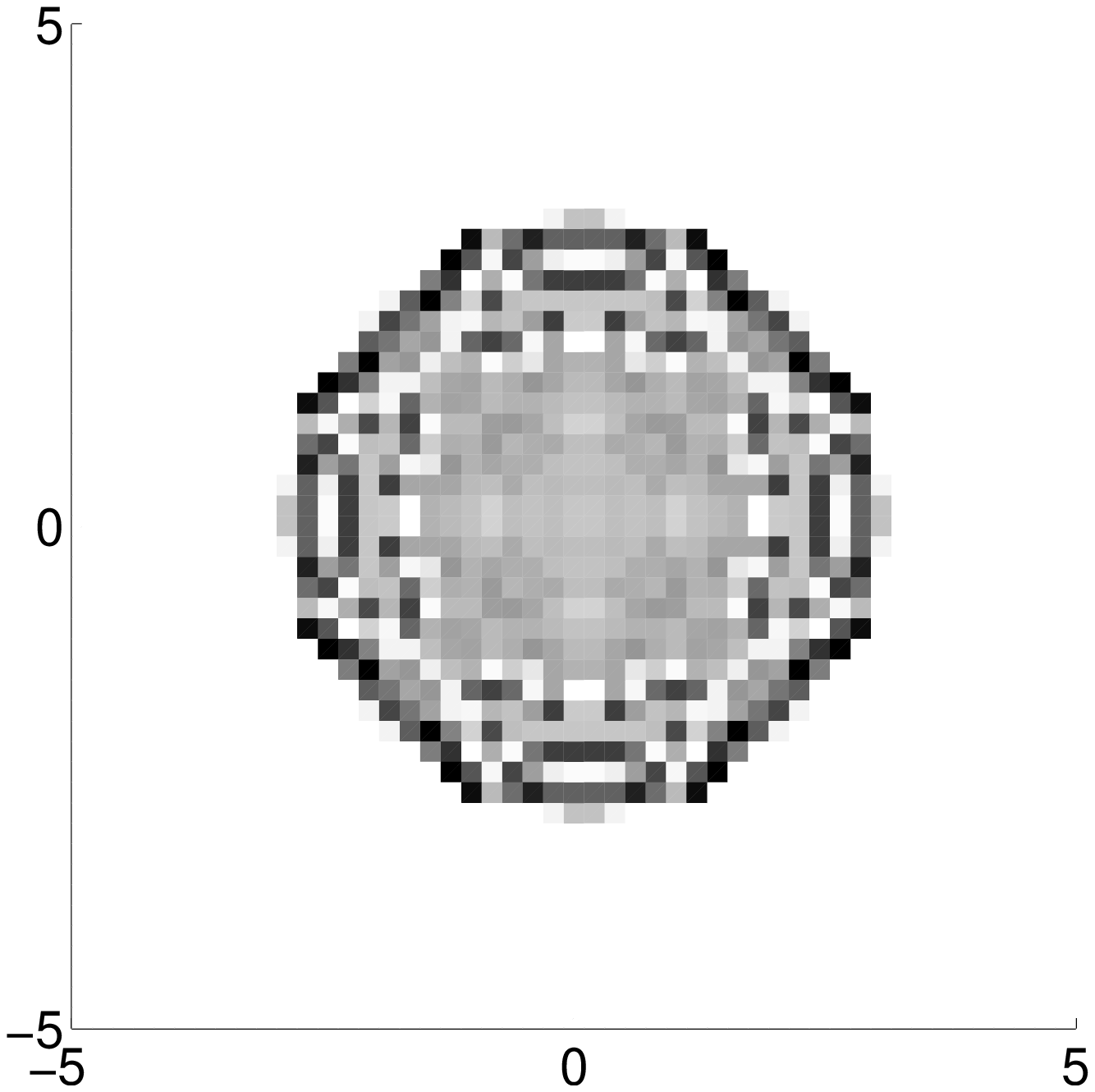}
	\includegraphics[width=.24\textwidth]{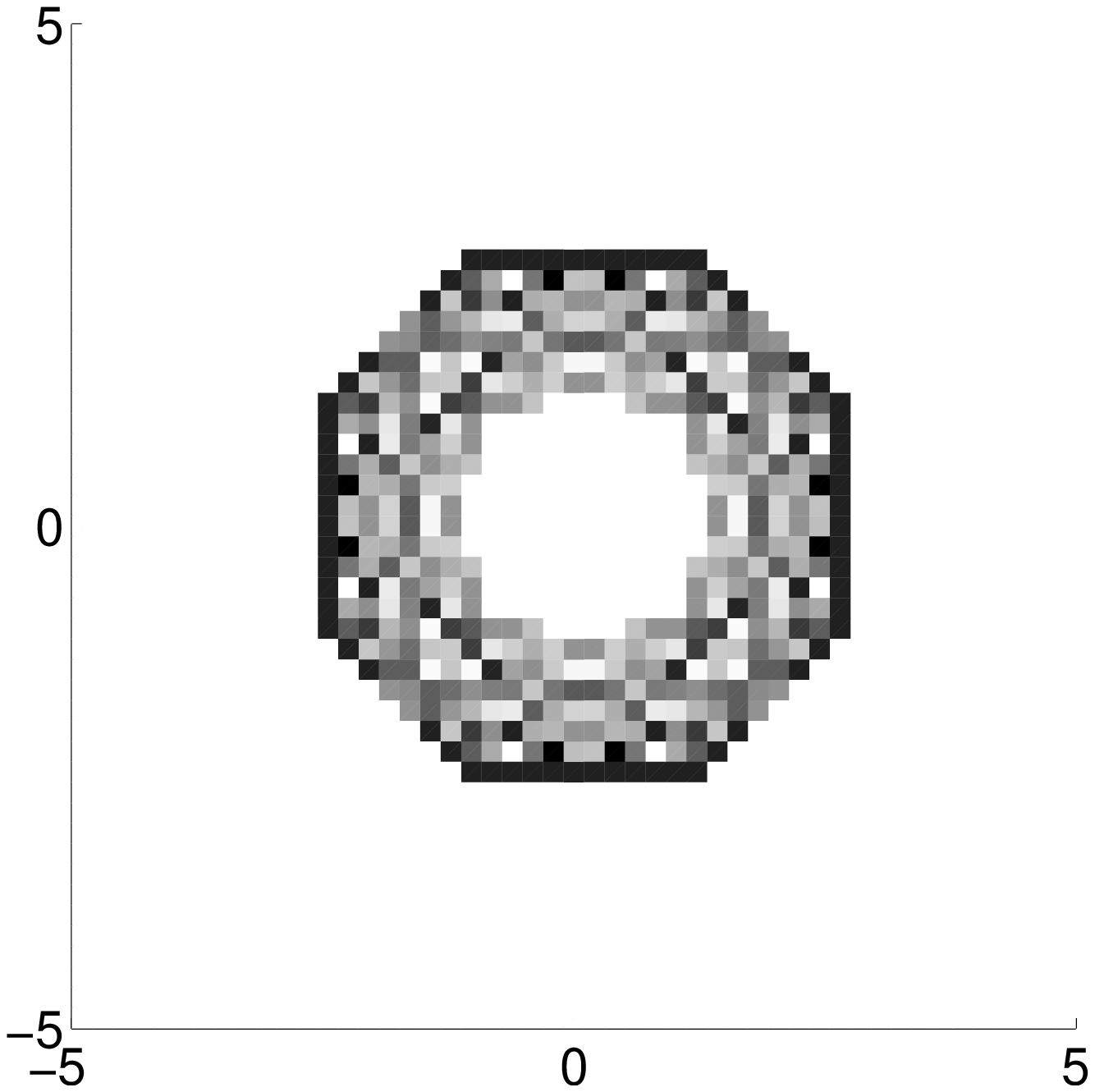}
	\caption{Barycenters of the measures $m_1,m_2, m_3,m_4$ (four rotated rectangles) with cost $C(x,y) = |x-y|^p$ for $p = 1, 2, 3$,
	using grid size $50^2$.	}
\label{fig:4meas}
\end{figure}

\subsection{Matching for teams}
We considered the matching for teams problem and used measures and costs as follows, also see Figure~\ref{fig:MatchTeams}.
Set  the quality domain  $Z= [0,1]^2$ and write $z = (z_1,z_2)$ for points in $Z$.
Set $M_0, M_1, M_2$ to be measures which have constant density on their support, and let their supports be
$[1,2]^2$,  $[1.25,1.75] \times [1,2]$, and  $[1,2]\times [1.25,1.75]$, respectively.
The corresponding cost functions (with the interpretation that $c_0$ is the negative of the buyer's utility) are 
\begin{align*}
c_0(x,z) &=  -5.5(x_1 z_1 + x_2 z_2)
\\
c_1(x,z) = c_2(x,z) &= (x_1+z_1)^2 +(x_2+z_2)^2.
\end{align*}

The solution concentrates mass at the boundary, and especially at the corners of the domain.
\begin{figure}[htdp]
	\centering
	\includegraphics[width=.49\textwidth]{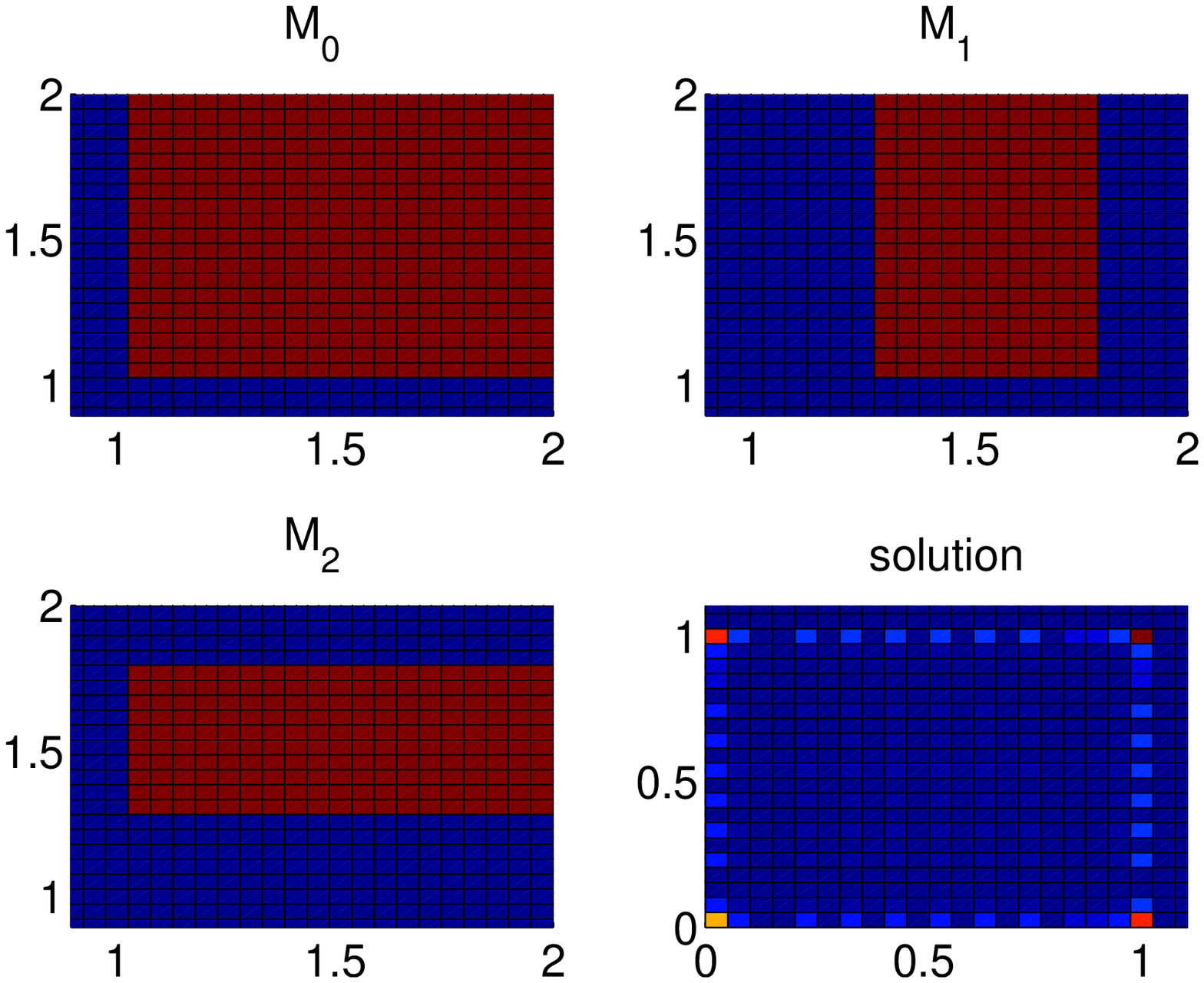}
	\includegraphics[width=.49\textwidth]{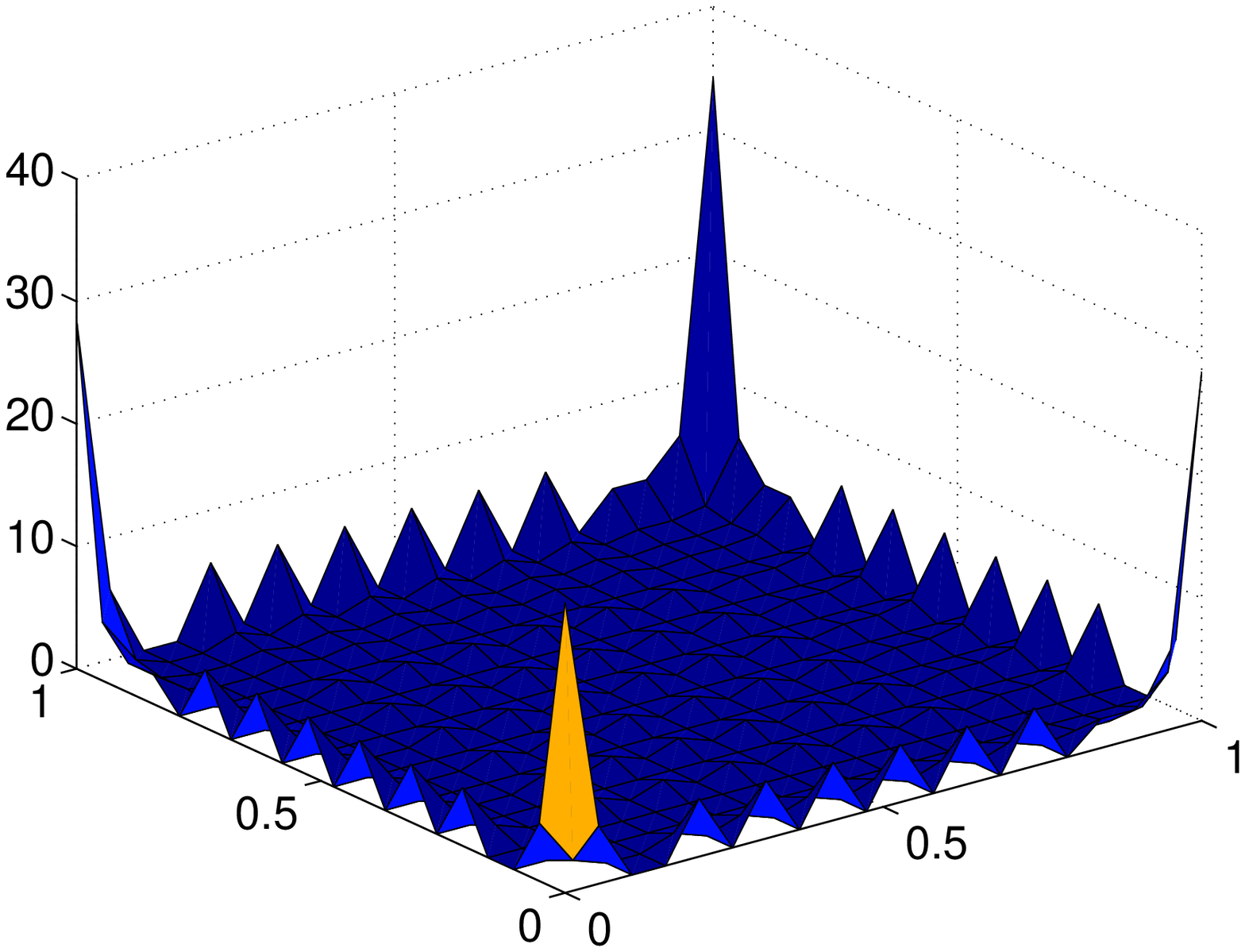}			
	\caption{Solution of the matching for teams problem.  Left: the three measures, and the solution.  Right: surface plot of the solution.  The solution concentrates mass mostly on the corners with some mass on the edges.}
\label{fig:MatchTeams}
\end{figure}


\label{numlp}


\section{Dual formulation}\label{duality}

\subsection{Duality and optimality conditions}

Let us now explain why  the variational problem \pref{dualmft} can be naturally be seen as a dual formulation of \pref{lpcouplings} (see \cite{ce} for more details on this duality). To that end, let us observe that $(\gamma_1, \ldots , \gamma_I)\in \Pi$ if and only if \pref{contrmu} holds for every $i$ (these are the fixed $\mu_i$ marginals constraints) and 
\begin{equation}\label{dualcommonmarginal}
 \sum_{i=1}^I \int_{X_i\times Z} \varphi_i (z) \gamma_i(dx_i, dz)=0, \mbox{ as soon as }  \sum_{i=1}^I \varphi_i(z)  =0, \forall z\in Z.
\end{equation}
Indeed, clearly if the $\gamma_i$'s have the same marginal on $Z$ then \pref{dualcommonmarginal} holds. Conversely assume \pref{dualcommonmarginal}, let $i\neq j$ and $\varphi \in C(Z)$ then applying \pref{dualcommonmarginal} to the potentials $\varphi_i=\varphi$, $\varphi_j=-\varphi$ and $\varphi_k=0$ for $k\in\{1,\ldots,I\}\setminus \{i, j\}$ we see that $\int_{X_i\times Z} \varphi (z) \gamma_i(dx_i, dz)=\int_{X_j\times Z} \varphi (z) \gamma_j(dx_j, dz)$. This proves that \pref{dualcommonmarginal} characterizes the fact that the $\gamma_i$'s share the same marginal on $Z$. This enables us to rewrite \pref{lpcouplings} in inf-sup form:
\begin{equation}\label{infsup}
\inf_{\gamma_i \ge 0} \sup \left\{ \Lag((\gamma_i)_i, (\psi_i)_i, (\varphi_i)_i ) \; : \;  \psi_i \in C(X_i), \; \varphi_i \in C(Z)\; : \; \sum_{i=1}^I \varphi_i =0  \right\}
\end{equation}
where the Lagrangian $\Lag$ is given by
\[\begin{split} \Lag((\gamma_i)_i, (\psi_i)_i, (\varphi_i)_i)&:=  \sum_{i=1}^I \int_{X_i\times Z} (c_i(x_i,z)-\psi_i(x_i)-\varphi_i(z)) \gamma_i(dx_i, dz)\\
&+\sum_{i=1}^I \int_{X_i} \psi_i(x_i) \mu_i(dx_i).
\end{split}\] 
To obtain the desired dual formulation, we formally switch the inf and the sup (again, we refer to \cite{ce} for a rigorous derivation):
\[\sup_{(\psi_i, \varphi_i), \; \sum \varphi_i=0} \inf_{\gamma_i\ge 0}  \Lag((\gamma_i)_i, (\psi_i)_i, (\varphi_i)_i). \]  
We next observe that 
\[\begin{split}
&\inf_{\gamma_i\ge 0}  \Lag((\gamma_i)_i, (\psi_i)_i, (\varphi_i)_i)=\sum_{i=1}^I \int_{X_i} \psi_i(x_i) \mu_i(dx_i)\\
&+\sum_{i=1}^I \inf_{\gamma_i\ge 0 }  \int_{X_i\times Z} (c_i(x_i,z)-\psi_i(x_i)-\varphi_i(z)) \gamma_i(dx_i, dz)
\end{split}\]
and the latter infimum is $0$ when 
\begin{equation}\label{constrdual}
c_i(x_i,z)\ge \psi_i(x_i)+\varphi_i(z), \; \forall (x_i, z)\in X_i\times Z
\end{equation}
and $-\infty$ otherwise. The dual of \pref{lpcouplings} therefore consists in  maximizing
\[\sum_{i=1}^I \int_{X_i} \psi_i(x_i) \mu_i(dx_i)\]
subject to the constraints \pref{constrdual} and $\sum_{i=1}^I \varphi_i =0$. For fixed $\varphi_i$, the maximal $\psi_i$ that satisfies \pref{constrdual} being $\psi_i=\varphi_i^{c_i}$, we see that that the dual can be equivalently formulated as 
\begin{equation}\label{dualcouplings}
\sup\left\{  \sum_{i=1}^{I}\int_{X_{i}}\varphi_{i}^{c_{i}}(x_i)%
\mu_{i}(dx_i)\mbox{  :  }\sum_{i=1}^{I}\varphi_{i}=0\right\}
\end{equation}  
which is exactly \pref{dualmft}. For the existence of solutions and the equality between the infimum in \pref{lpcouplings} and the supremum in \pref{dualcouplings} (which is obtained by a slightly different argument), we again refer to \cite{ce}.  Now the optimality conditions for \pref{lpcouplings} and \pref{dualcouplings} are summarized by the equivalence between the following assertions:

\begin{itemize}

\item $(\gamma_i)_i\in \Pi$ solves \pref{lpcouplings} and $(\varphi_i)_i$ such that $\sum_{i=1}^I \varphi_i =0$ solves \pref{dualcouplings},

\item $((\gamma_i)_i, (\varphi_i^{c_i})_i, (\varphi_i)_i)$ is a saddle point of $\Lag$,

\item for every $i$, one has
\begin{equation}\label{saturopt}
\varphi_i^{c_i}(x_i)+\varphi_i(z)=c_i(x_i, z)
\end{equation}
$\gamma_i$-almost everywhere on $X_i\times Z$ or, equivalently, by continuity, on the support of $\gamma_i$.

\end{itemize}

\subsection{The case of Wasserstein barycenters}

From now on, we restrict ourselves to the quadratic case where all the $X_i$'s and $Z$ are some ball $B$ (say) of $\R^d$ and the costs $c_i$ are quadratic:
\[c_i(x_i, z):=\frac{\lambda_i}{2} \vert x_i -z\vert^2\]
where the $\lambda_i$'s are positive coefficients which we normalize in such a way that they sum to $1$.  In this case, \pref{primalmft} corresponds to
\begin{equation}\label{barycenterpbm}
\inf_{\nu\in \PP(B)} \sum_{i=1}^I \lambda_i W_2^2(\mu_i, \nu)
\end{equation}
where $W_2^2$ stands for the squared $2$-Wasserstein distance. This problem has been studied in details in \cite{ac} where uniqueness (under the assumption that one of the measures does not give mass to small sets), characterization,  $L^p$ or $L^{\infty}$ regularity   results are established for Wasserstein barycenters  as well as a close connection with the quadratic multimarginal optimal transport problem  of Gangbo and \'Swi\c ech \cite{gansw}. Since  Wasserstein barycenters constitute a natural way to interpolate between an arbitrary number of measures, they therefore also find applications in image processing \cite{peyre} and statistics \cite{bk}.

Let us now informally give the optimality conditions for the barycenter using once again the dual formulation \pref{dualcouplings} (see  \cite{ac} for details). In the present quadratic setting, the formula for the $c_i$-transform takes the form
\[\varphi_i^{c_i}(x_i)=\inf_{z} \left\{ \frac{\lambda_i}{2}\vert x_i-z\vert^2 -\varphi_i(z)  \right\},\]
which, defining
\[u_i(x_i):=\frac{1}{2} \vert x_i \vert^2 -\frac{\varphi_i^{c_i}(x_i)}{\lambda_i}\]
can conveniently be rewritten as
\[u_i=\Big(  \frac{1}{2} \vert . \vert^2- \frac{\varphi_i}{\lambda_i} \Big)^*\]
(where $*$ denotes the usual Legendre transform). In particular, $u_i$ is convex (hence differentiable outside of a small set) and defining $v_i:=u_i^*$ we have
\begin{equation}\label{ineg23}
 \frac{1}{2} \vert . \vert^2- \frac{\varphi_i}{\lambda_i} \ge \Big(  \frac{1}{2} \vert . \vert^2- \frac{\varphi_i}{\lambda_i} \Big)^{**}=u_i^*=v_i.
\end{equation}
Moreover the optimal coupling $\gamma_i$ is concentrated on the set where equality \pref{saturopt} holds which is equivalent to the relation $u_i(x_i)+  \frac{1}{2} \vert z \vert^2- \frac{\varphi_i(z)}{\lambda_i} =x_i \cdot z$ but recalling \eqref{ineg23}, this implies  that $x_i \cdot z \ge u_i(x_i)+v_i(z)=u_i(x_i)+u_i^*(z) \ge x_i \cdot z$ so that $z=\nabla u_i(x_i)$ (provided $u_i$ is differentiable at $x_i$ which is the case $\mu_i$ a.e. as soon as $\mu_i$ does not charge small sets...). This implies that the barycenter which is also the marginal $\nu$ that is common to all the $\gamma_i$'s is actually given by $\nu=\nabla {u_{i}}_{\#} \mu_i$ for every $i$ and $\nabla u_i$ is the optimal transport between $\mu_i$ and $\nu$ for $W_2^2$. 
As explained above, we can deduce from \pref{ineg23} and the fact that $\gamma_i$-almost everywhere equality \pref{saturopt} holds    that for $\nu$-a.e. $z$, one has 
\[
 \frac{1}{2} \vert z \vert^2- \frac{\varphi_i(z)}{\lambda_i}=v_i(z).
\]
 Recalling that the $\varphi_i$ have to sum to $0$, we deduce that
 \begin{equation}\label{sumid}
 \sum_{i=1}^I \lambda_i v_i(z)=\frac{\vert z\vert^2}{2}
 \end{equation}
 on the support of $\nu$. The optimality conditions for the barycenter $\nu=\nabla {u_{i}}_{\#} \mu_i=\nabla {v_{i}}^*_{\#} \mu_i$ therefore, at least formally take the form of the system of Monge-Amp\`ere equations
\begin{equation}\label{MAsystem}
\nu=\mu_i(\nabla v_i) \det(D^2 v_i), \; i=1,\ldots, I
\end{equation}
which is supplemented with equation \pref{sumid} on the support of $\nu$. We shall see in the next paragraph how to compute numerically in an efficient way the potentials $\varphi_i$.


\subsection{An efficient algorithm for Wasserstein barycenters}\label{sec:algobar}

\paragraph{Discretization of the dual problem.} We assume in all this section that the sets  $X_i$'s and $Z$ are subsets of $\R^d$ for some $d=1,2$ or $3$.  As described in the previous section, the computation of one Wasserstein quadratic barycenter is equivalent in its dual form to the maximization of
\begin{equation}
\label{dualcouplings2}
\sum_{i=1}^{I}\int_{X_{i}}\varphi_{i}^{c_{i}}(x_i) \mu_{i}(dx_i)
\end{equation}  
where
\[\varphi_i^{c_i}(x_i)=\inf_{z} \left\{ \frac{\lambda_i}{2}\vert x_i-z\vert^2 -\varphi_i(z)  \right\},\]
under the linear constraint $\sum_{i=1}^{I}\varphi_{i}(z)=0$ for all $z\in Z$. 
This formulation leads to the following natural discretization of Wasserstein quadratic barycenter computation. Suppose $(y_i^j,\nu_i^j)_{j = 1,\dots, N_i} \subset X_i \times \R_+$    is a convergent quantization of the measures $ \mu_{i}$. More explicitly, we assume that for all $i=1,\dots,I$ 
$$\lim_{N_i \rightarrow \infty} c(N_i)  \sum_{j=1}^{N_i} \nu_i^j \delta_{y_i^j} = \mu_{i}$$
in the sense of the weak convergence of measures. In order to only consider probability measures, we set $c(N_i)= (\sum_{j=1}^{N_i} \nu_i^j)^{-1}$. Additionally, we suppose that $(z_k)$ is a dense countable family of points of $Z$.  Based on (\ref{dualcouplings2}) and previous quantizations, for a given $N_k \in \N$, our discrete optimization problem of $I\times N_k$ variables reads
\begin{equation}
\label{dualcouplings_d}
\Phi((\varphi_{1}^{k}),\dots,(\varphi_{I}^{k})) = \sum_{i=1}^{I} c(N_i) \sum_{j=1}^{N_i} \nu_i^j \min_{k=1,\dots,N_k} \left\{ \frac{\lambda_i}{2}\vert y_i^j-z_k\vert^2 -\varphi_i^k  \right\}
\end{equation}
under the $N_k$ pointwise linear constraints: 
\begin{equation}
\label{dualcouplings_constr_d}
\sum_i \varphi_{i}^{k} = 0, \qquad \forall k=1,\dots ,N_k.
\end{equation}
This optimization problem in its dual form can be seen as a large scale non-smooth concave maximization problem. We discuss in the next paragraph  alternatives that have been developed  to solve numerically this type of problems. 

\paragraph{Non smooth algorithms.}  Many different approaches
have been introduced  in the last decades to approximate optimal solution of
non-smooth concave (or convex) problems, e.g.  gradient sampling
methods \cite{burke2005robust} and bundle methods
\cite{lemarechal1978nonsmooth}. These algorithms make use of a partial
or complete description of superdifferentials  in order to identify
ascent directions (see next paragraph). Even though Proposition~\ref{prop:subdifferential}
describes explicitly the whole superdifferential, finding an effective
ascent direction in practice is made difficult by the large dimension
of some superdifferentials.  Additionally,
those approaches are essentially of order one and follow the singular
parts of the graph of the cost function. These two facts could explain
a slow rate of convergence when starting from an initial point far
away from any optimal vector.

One surprisingly efficient alternative for minimizing non-smooth
functions is the use of quasi-Newton methods. It is known
\cite{powell1976some} that if the maximized function, $\Phi$, is twice
continuously differentiable and the suplevel set $\Phi \geq \Phi(x_0)$
is bounded, then the sequence of function values generated by the BFGS
method with inexact Armijo-Wolfe line search, starting from $x_0$
converge to the maximal value of $\Phi$. More recently, it has been
pointed out by different authors \cite{urruty1996convex,lewis2009nonsmooth,lewis2008behavior} that
variable metric algorithms may produce in some cases sequences which
converge to an optimal point in the sense of Clarke. The mathematical
analysis of this good behavior has just been initiated in recent
papers of Overton \cite{lewis2009nonsmooth,lewis2008behavior}. This
efficiency could be explained heuristically by the fact that the
approximated inverse of the Hessian matrix has a spectral
decomposition in two subspaces which describe the two different local
behaviors of the cost $\Phi$: a subspace associated to the regular
directions of the cost function $\Phi$, and the subspace of
eigenvectors whose eigenvalues are small in absolute value, which
correspond to the singular directions of $\Phi$.

It has been observed in simple situations that L-BFGS (low memory version of
Broyden-Fletcher-Goldfarb-Shanno algorithm) algorithms are sometimes
able to converge to an optimal point. In more standard examples, where concentration can occur for instance, L-BFGS
approach fails to converge. This expected behavior
for strongly non-smooth functions illustrates the need for using more specific non-smooth techniques close from the optimal point. The costly, but reliable, bundle type algorithms have demonstrated their efficiency in this context.

We will not give here a detailed study of quasi Newton methods applied
to optimal transportation which would be out of the scope of this
paper. We only point out that the L-BFGS  algorithm combined  with a bundle approach
 gives a rather efficient way to solve this type of problem. We refer to \cite{haarala2007globally} for a careful study and an efficient implementation of this kind of hybrid algorithm.

\paragraph{Gradient computation.} The previous approach relies on the capability of providing at every iteration one supergradient vector of the current  iterate. It is straightforward to obtain the following characterization of the supergradient of the discrete dual cost $\Phi$:

\begin{prop}
\label{prop:subdifferential}
Let  $(\varphi_{1}^{k}),\dots,(\varphi_{I}^{k})\in \R^{N_k \times I}$. Then 
$$((v_1^k),\dots,(v_I^k)) \in \partial \Phi ((\varphi_{1}^{k}),\dots,(\varphi_{I}^{k}))$$
if and only if it is a convex combination of the finite set of extremal vectors defined in the following way.
Let $\varphi_i^{k(i,j)^*}$ be any \emph{selection} of minimizing values involved in the dual cost. That is $\forall i,j$
\begin{equation}
\label{eqargmin}
 k(i,j)^* \in \argmin_{k=1,\dots,N_k} \left\{ \frac{\lambda_i}{2}\vert y_i^j-z_k\vert^2 -\varphi_i^k  \right\}.
\end{equation}
Then, the set of extremal vectors is the finite collection of all vectors 
\begin{equation}
\displaystyle (e^{k^*}) =\displaystyle (\sum_j \sum_{k} -c(N_i)\nu_i^j \delta_{k(i,j)^*}(k))
\end{equation}
for any \emph{selection} $(k(i,j)^*)$.
\end{prop}
A crucial observation that has been raised in \cite{mo2012discrete} is the fact that the computation of a vector of the superdifferential does not require generically an order of $I \times N_k \times \sum N_i$ operations. Actually the following formulation makes it possible to use a special data structure called \emph{kd-tree} which in most cases reduces the complexity of finding one vector of $\partial \Phi ((\varphi_{1}^{k}),\dots,(\varphi_{I}^{k}))$. Notice that for very large scale problems, the so called \emph{Approximate Nearest Neighbor Search} could provide a relevant tool to relax our problem. In all our experiments we performed  exact searches.

Let $i,j$ be given and assume we want to evaluate the minimal value
$$M = \min_{k} \left\{ \frac{\lambda_i}{2}\vert y_i^j-z_k\vert^2 -\varphi_i^k  \right\}.$$
Let us then define $c_i=\max_k \varphi_i^k $ and
$$ M = -c_i + \min_{k} \left\{ \frac{\lambda_i}{2}\vert y_i^j-z_k\vert^2 + c_i -\varphi_i^k  \right\}.$$
Since the latter term is positive:
$$M=-c_i + \min_k || P_i^j - Q_i^k ||^2$$
where $\displaystyle P_i^j = (\sqrt{\frac{\lambda_i}{2}} y_i^j,0)$ is a fixed vector of $\R^{d+1}$, $\displaystyle Q_i^k = (\sqrt{\frac{\lambda_i}{2}}z_k,\sqrt{c_i -\varphi_i^k})$ and $||.||$ stands for the standard euclidean norm of $\R^{d+1}$. Thus our supergradient request reduces to identify one closest point of $P_i^j$ among points of $(Q_i^k)_k$. Observe that the family  $(Q_i^k)_k$ does not depend of the parameter $j$. This task is a standard operation in computational geometry which can be performed efficiently with  \emph{kd-tree} structures. By using such tools, we can reduce the complexity of the supergradient request in the generic case to an order of  $I \times \log N_k \times \sum N_i $ operations. Observe that if the $N_i$'s and $N_k$ are of  same order $N$, one request is generically of complexity of order $I^2 N \log N$.

\paragraph{Numerical quantization and localization.} 
We suppose that all the measures $\mu_i$ are compactly supported. 
We use the discretization of Section~\ref{sec:discretization}.  
The support of the unknown barycenter measure is bounded using the results of Section~\ref{sec:localization}, in particular, \eqref{localizbar}

\paragraph{Reconstruction of the barycenter density.} One additional difficulty associated to the dual formulation is the fact that optimal dual vectors only give an implicit description of the barycentric measure. In order to recover the support and the density of the barycentric measure, we introduce the following least square procedure. 

Every optimal dual vector  $(\varphi_i^k)_k$ must be associated to an optimal transport from $\mu_i$ to the barycentric measure. A crucial observation is the fact that every associated map transports the  $\mu_i$ to the same measure. We exploit this optimality condition to recover the barycentric measure through the coefficients $f_i^{j,k}$ which describe the mass transported from $y_i^j$  to $z_k$. By optimality, some mass can be transported from $y_i^j$  to $z_k$ if and only if 
\begin{figure}[htbp]
\centering
\begin{tabular}{c c}
\includegraphics[width=.39\textwidth]{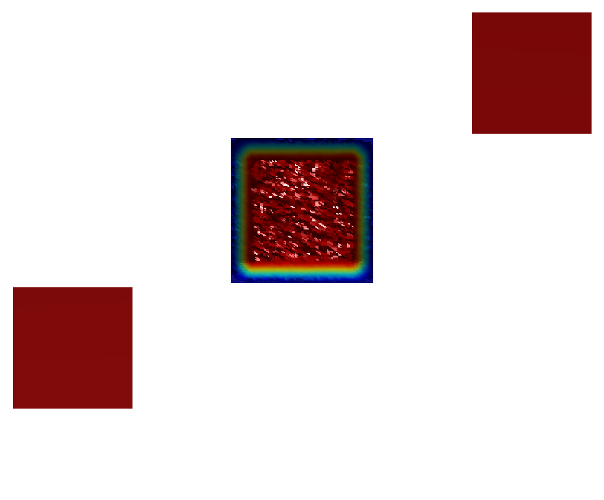}
&	\includegraphics[width=.49\textwidth]{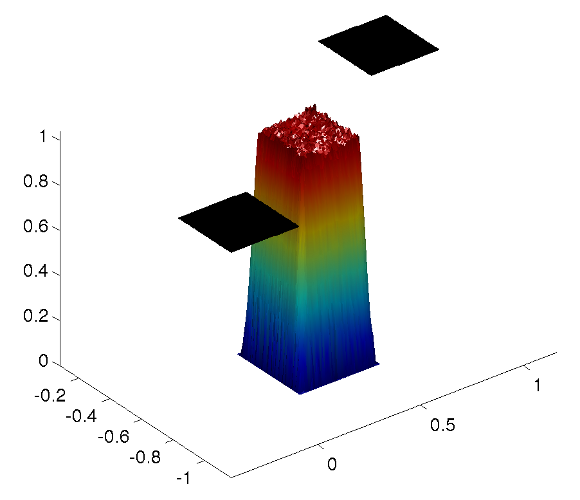}
\end{tabular}
\caption{Classical McCann interpolation between translated measures}
\label{fig:bary_dual_0}
\end{figure}
\begin{equation}
\label{vopt}
G_{i,j}(\varphi_i^{k}) =  \min_{k'=1,\dots,N_k} G_{i,j}(\varphi_i^{k'}).
\end{equation}
where $G_{i,j}(\varphi_i^{k'}) = \frac{\lambda_i}{2}\vert y_i^j-z_k\vert^2 -\varphi_i^{k'} $. Let us fix some parameter $\varepsilon>0$. Based on the  previous observation, we only consider the unknown coefficients $f_i^{j,k}$ for which $G_{i,j}(\varphi_i^{k})$ is less than the optimal value  \eqref{vopt} plus $\varepsilon$.  In order to recover the barycentric measure, we look for the set of coefficients which generate the same measures in an optimal least square sense. More precisely, we solve the sparse least square   problem:
\begin{equation}
\label{optrecover}
\min_{f_i^{j,k}} \sum_{l,m} \sum_k (\sum_j f_l^{j,k} \nu_l^j - \sum_{j'} f_m^{j',k} \nu_m^{j'} )^2
\end{equation}
among non negative coefficients less than one which satisfy the linear constraints 
$$\forall i,j,\,  \sum_k f_i^{j,k} = 1.$$


\subsection{Numerical results}\label{sec:Num2}

 As  detailed in the previous paragraphs, our approach relies first on a non smooth optimization step using an hybrid LBFGS/Bundle algorithm and a fast computation of supergradient vectors. In a second phase, a sparse least square problem is solved in order to recover an approximation of barycentric density. Let us point out that in all the following examples, the first optimization step was stopped after one hour of computation on a standard PC. This costly step could be dramatically sped up in using a straightforward parallelized cost evaluation.

We validate our approach by considering different test cases for which analytic descriptions of barycenters are available. The simplest situation is the case of the barycenters of a measure of density  $\rho(.)$ and a translated measure of density $\rho(. + V)$ where $V$ is some fixed vector.  In this trivial case, the isobarycentric measure is of course the measure of density $\rho(. +V/2)$. We display in figure \ref{fig:bary_dual_0}, the barycentric measure obtained by our located approximation scheme applied to $\rho = \chi_c$ where $\chi_c$ is the characteristic function of a unit square. In this experiment, we used a grid of size $200 \times 200$ and a recovering parameter $\varepsilon = 10^{-5}$. In the least square optimization problem \eqref{optrecover}, we obtain an error of order $ 10^{-4}$ for every quadratic term.

Next we applied our approach to the case of Gaussian measures. A complete description of barycenters of Gaussian measures has been given in \cite{ac}: consider a family of Gaussian measures $\mu_i(m_i,S_i)$ 
of means $(m_i)_i$ and covariance matrices $(S_i)_i$. Then, the barycentric measure  associated to the non-negative weights $(\lambda_i)_i$ is a Gaussian measure of mean the barycenter of the $(m_i,\lambda_i)_i$. Moreover, its covariance matrix is the only definite positive matrix $S$ solution of the equation
\begin{equation}
\label{covmatrix_eq}
\sum_i \lambda_i (S^{1/2}S_iS^{1/2})^{1/2}=S.
\end{equation}

\begin{figure}[htbp]
\centering  
\begin{tabular}{c c}
\includegraphics[width=.29\textwidth]{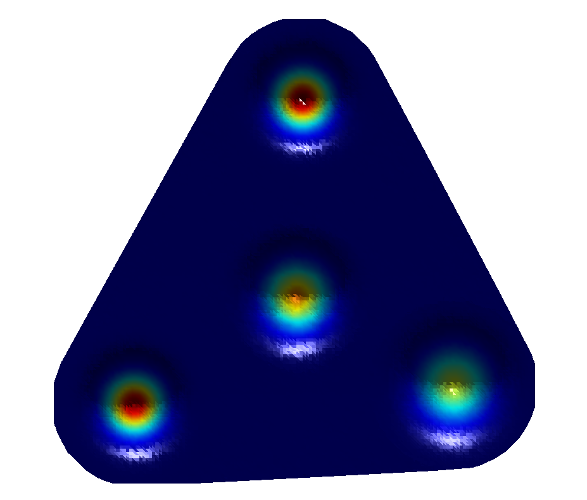}
&	\includegraphics[width=.29\textwidth]{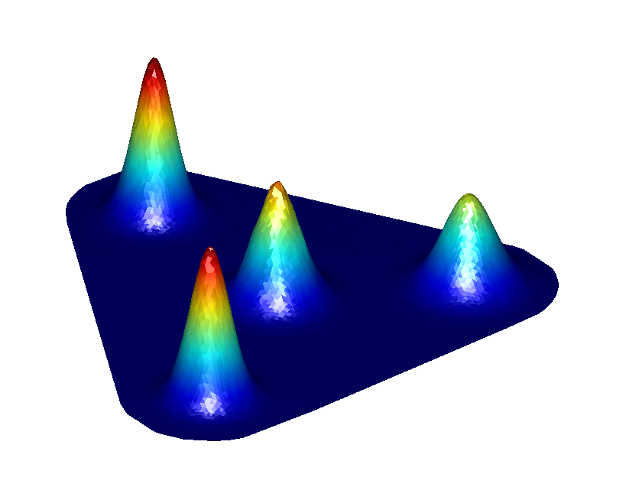}\\
	\includegraphics[width=.29\textwidth]{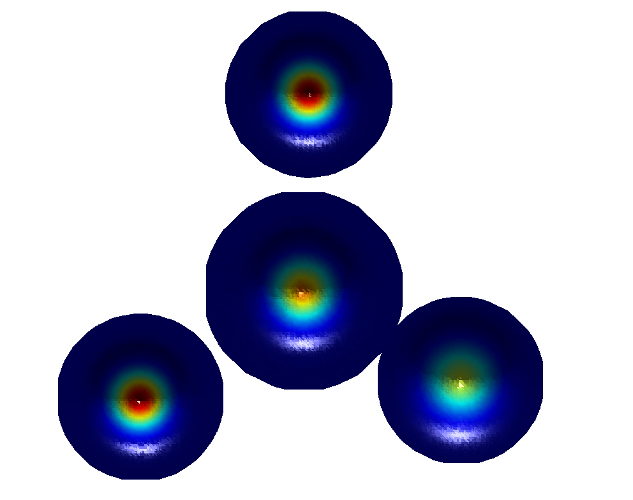}
&	\includegraphics[width=.29\textwidth]{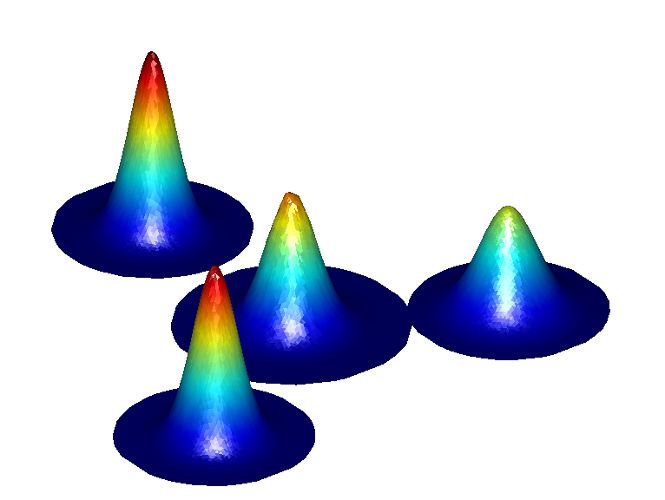}
\end{tabular}
\caption{Isobarycenter computation of three gaussian measures by a global (first row) and a localized approach (second row).}
\label{fig:bary_dual_1}
\end{figure}

We denote by $\mathcal{N}(m_i,\sigma_i)$ a Gaussian of mean $m_i$ and of covariance matrix equal to $\sigma_i^2 Id$. We considered two different test cases and applied for both our global and localized approaches. In all the experiments the number of sampling points of the given measures and of the barycentric measure have been fixed for the global approach to $\forall i,\, N_i=N_z = 15 \times 10^3$. For the localized approach by Minkowski sum, we imposed $\forall i,\, N_i=N_z = 5 \times 10^3$.
The first test case consists in approximating the isobarycenter of the three Gaussian measures  of random standard variations $\mathcal{N}((0.1,0.8),1/49.75)$,  $\mathcal{N}((-0.9,-1),1/35.89)$ and $\mathcal{N}((1,-0.9),1/74.63)$. More precisely, due to the infinite support of Gaussian measures, we apply the following threshold: for every Gaussian measure, we restrict the support to the grid point contained in a unit disk centered at the  mean vector. Thus, we apply a uniform normalization to obtain measures of the same total mass.
  
The resulting barycenter and the given measures are presented in figure \ref{fig:bary_dual_1}. Our second test case is related to the approximation of the barycenter of the five gaussian measures $(\mathcal{N}(m_i,\sigma_i),\lambda_i)_{i=1,\dots,5}$ where the $m_i$ are the vertices  of  a regular pentagon with $\sigma_i=1/50$ and $\lambda_i=1/7$ if $i$ is odd and $\sigma_i=1/100$ and $\lambda_i=2/7$ otherwise. The resulting barycenter and the given measures are drawn in figure \ref{fig:bary_dual_2}. Observe that in the localized illustrations, the support of the unknown measure is not anymore centered due to the loss of symmetry in Minkowski's sum. To conclude the study of those test cases, we provide in table \ref{table:dis}, the errors between the theoretical and computed means and covariance matrices. As expected, even if the number of degree of freedom is smaller, the results obtained by the located approach are significantly better than the ones obtained by the first algorithm.

\begin{table}[b]
\begin{center}
\begin{tabular}{l|c|c|}
                               & $||m_i^{th} - m_i||$    & $|\sigma_i^{th} - \sigma_i|$\\  
\hline  
First test case                &  $ 0.003 $       & $ 0.005$         \\ 
First test case localized      &  $ 0.0002$       & $ 0.0014$        \\ 
Second test case               &  $ 0.002$        & $ 0.07$          \\ 
Second test case localized     &  $ 0.0002$       & $ 0.0013$        \\ 
\end{tabular}
\caption{Upper bounds of the gap between  theoretical and computed means and covariance coefficients}
\end{center}
\label{table:dis}
\end{table}

To conclude our numerical experiments, we provide large scale examples in which we interpolate three textures of images of size  $150 \times 150$. This type of applications have been first studied in the framework of optimal transportation in \cite{peyre} (also see Galerne et al. \cite{ggm} for a different setting using the Fourier spectrum that is useful in the case of color images). The texture mixing problem consists in synthesizing a texture from a family of given textures. The interest of using Wasserstein barycenters in this context is due to the spatial nature of Wasserstein distance which provides a more natural interpolation process than the naive pointwise means (see figures \ref{fig:texture1} and \ref{fig:texture2}). We carried out similar experiments as the one depicted in \cite{peyre}. Our contribution with respect to \cite{peyre} lies in the fact that we do not replace the quadratic Wasserstein distance  by  the easier to handle so-called sliced Wasserstein distance (which is an average over directions of  one dimensional Wasserstein distances). We obtained by our method an approximation of the original model up to an error of $10^{-3}$ for every quadratic term of \eqref{optrecover}.


\begin{figure}[htbp]
\centering
\begin{tabular}{c c}
\includegraphics[width=.29\textwidth]{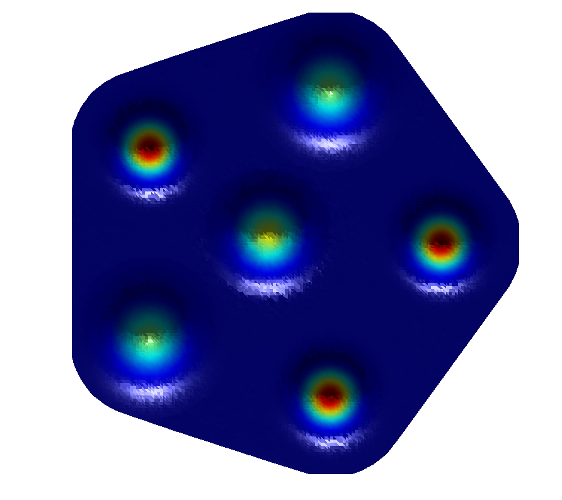}
&	\includegraphics[width=.29\textwidth]{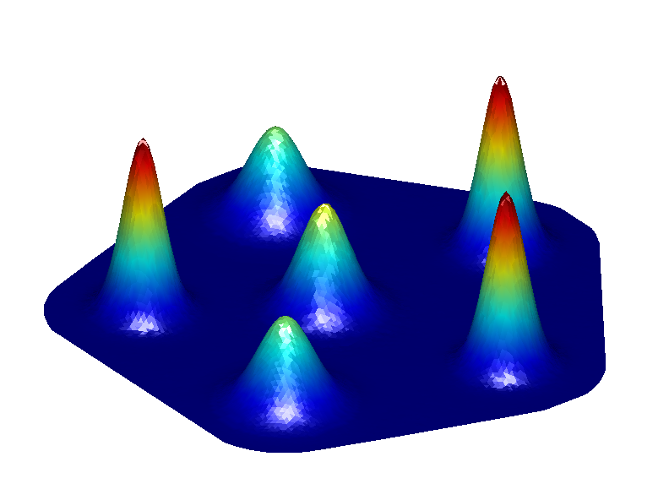}\\
	\includegraphics[width=.29\textwidth]{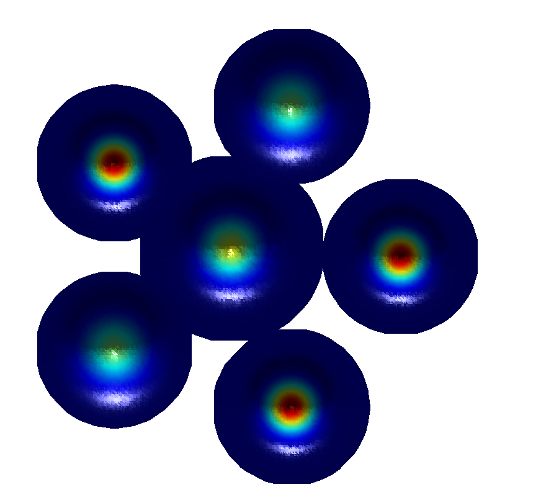}
&	\includegraphics[width=.29\textwidth]{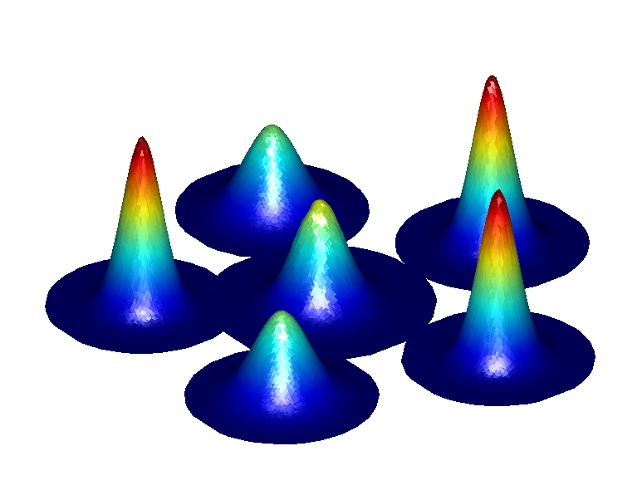}
\end{tabular}
\caption{Non uniform barycenter computation of five gaussian measures by a global (first row) and a localized approach (second row).}
\label{fig:bary_dual_2}
\end{figure}

\begin{figure}[htbp]
\centering
\begin{tabular}{c c c}
	\includegraphics[width=.29\textwidth]{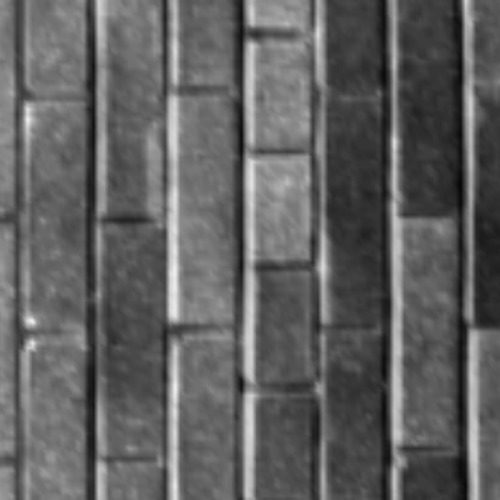}
&	\includegraphics[width=.29\textwidth]{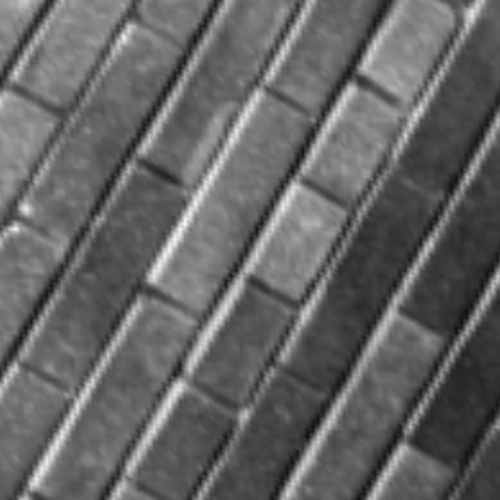}		
&	\includegraphics[width=.29\textwidth]{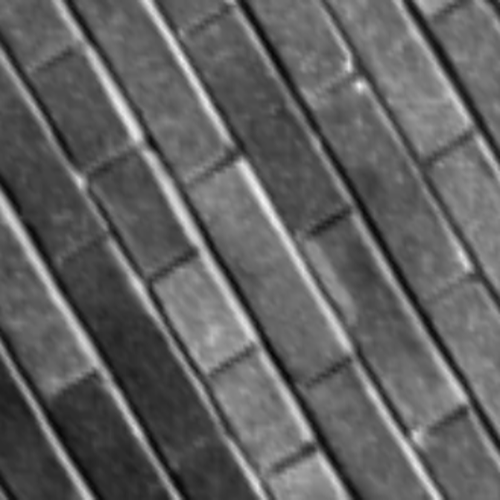}
\end{tabular}
\begin{tabular}{c c}
\includegraphics[width=.29\textwidth]{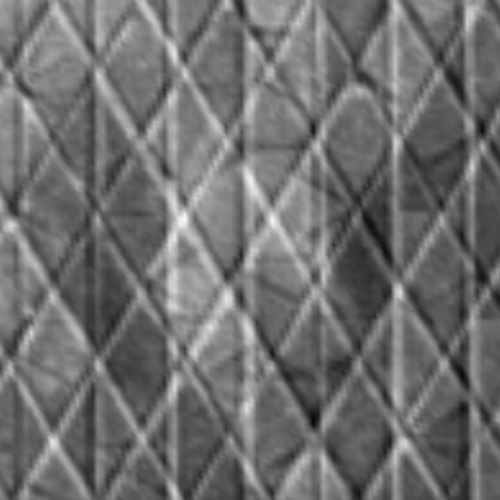}
&\includegraphics[width=.29\textwidth]{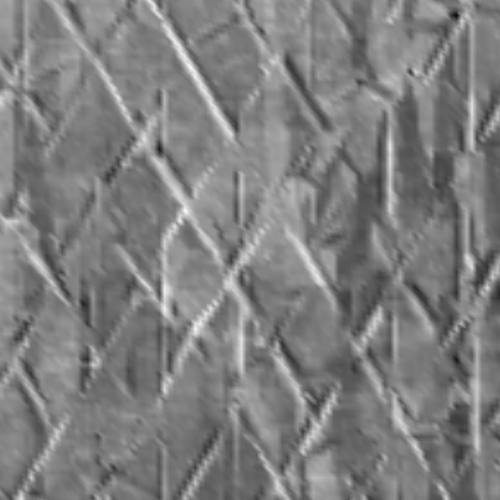}
\end{tabular}
\caption{Isobarycenter of the three textures of the first row. The pointwise mean of the three textures corresponds to the left picture of the second row. Wasserstein barycenter is presented on the right picture of the second row.}
\label{fig:texture1}
\end{figure}

\begin{figure}[htbp]
\centering
\begin{tabular}{c c c}
	\includegraphics[width=.29\textwidth]{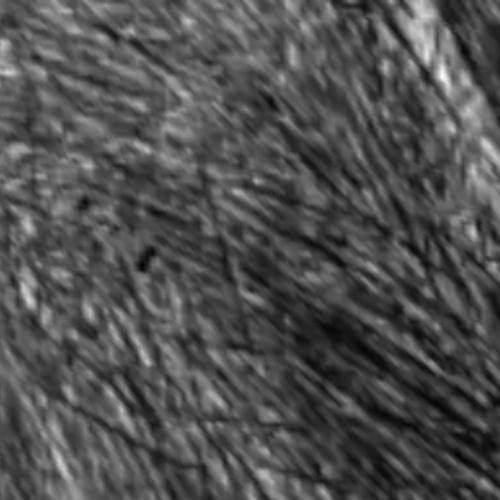}
&	\includegraphics[width=.29\textwidth]{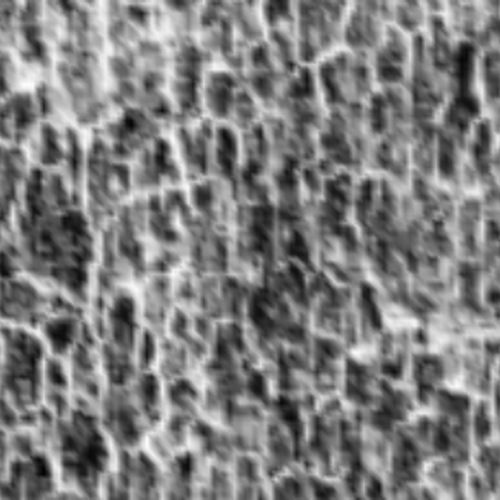}		
&	\includegraphics[width=.29\textwidth]{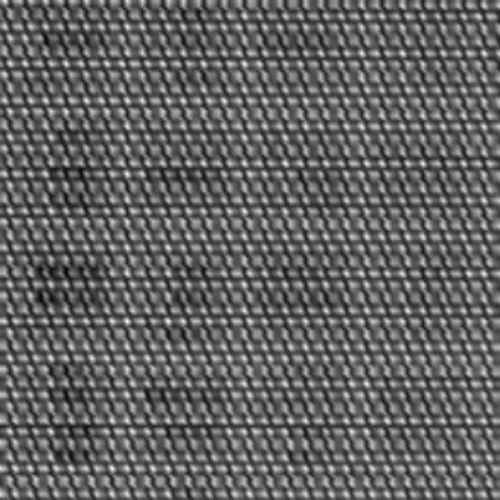}
\end{tabular}
\begin{tabular}{c  c}
\includegraphics[width=.29\textwidth]{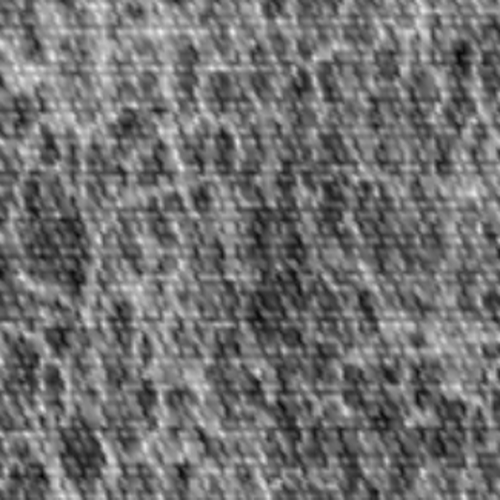}
&\includegraphics[width=.29\textwidth]{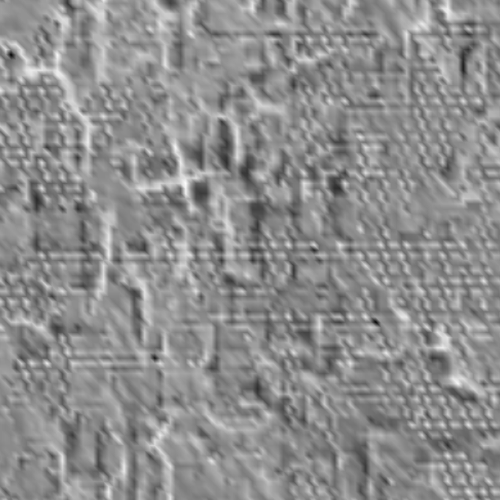}
\end{tabular}
\caption{Isobarycenter of the three textures of the first row. The pointwise mean of the three textures corresponds to the left picture of the second row. Wasserstein barycenter is presented on the right picture of the second row.}
\label{fig:texture2}
\end{figure}



\smallskip 

\textbf{Acknowledgements:} The authors are grateful to the hospitality of BIRS where the present work was initiated at the occasion on a focused group meeting on Numerical methods for optimal transport.  They are happy to thank Martial Agueh, Jean-David Benamou and Brendan Pass for many fruitful conversations. The first author gratefully acknowledges the support of the ANR, through the projects ISOTACE and OPTIFORM and INRIA through the "action exploratoire" MOKAPLAN.

\bibliographystyle{alpha}
\bibliography{BaryCenter}

\begin{thebibliography}{LMfASA78}

\bibitem[AC11]{ac}
Martial Agueh and Guillaume Carlier.
\newblock Barycenters in the wasserstein space.
\newblock {\em SIAM Journal on Mathematical Analysis}, 43(2):904--924, 2011.

\bibitem[BB00]{bb}
Jean-David Benamou and Yann Brenier.
\newblock A computational fluid mechanics solution to the {M}onge-{K}antorovich
  mass transfer problem.
\newblock {\em Numer. Math.}, 84(3):375--393, 2000.

\bibitem[BFO14]{benamou2014numerical}
Jean-David Benamou, Brittany~D Froese, and Adam~M Oberman.
\newblock Numerical solution of the optimal transportation problem using the
  monge--amp{\`e}re equation.
\newblock {\em Journal of Computational Physics}, 260:107--126, 2014.

\bibitem[BK12]{bk}
J{\'e}r{\'e}mie Bigot and Thierry Klein.
\newblock Consistent estimation of a population barycenter in the wasserstein
  space.
\newblock {\em arXiv preprint arXiv:1212.2562}, 2012.

\bibitem[BLO05]{burke2005robust}
J.V. Burke, A.S. Lewis, and M.L. Overton.
\newblock A robust gradient sampling algorithm for nonsmooth, nonconvex
  optimization.
\newblock {\em SIAM Journal on Optimization}, 15(3):751--779, 2005.

\bibitem[Bre91]{Brenier}
Yann Brenier.
\newblock Polar factorization and monotone rearrangement of vector-valued
  functions.
\newblock {\em Comm. Pure Appl. Math.}, 44(4):375--417, 1991.

\bibitem[CD14]{cuturidoucet}
Marco Cuturi and Arnaud Doucet.
\newblock Fast computation of {W}asserstein barycenters.
\newblock Proceedings of the 31st International Conference on Machine Learning
  (ICML-14), pages 685--693. 2014.

\bibitem[CE10]{ce}
Guillaume Carlier and Ivar Ekeland.
\newblock Matching for teams.
\newblock {\em Economic Theory}, 42(2):397--418, 2010.

\bibitem[GB08]{gb08}
Michael Grant and Stephen Boyd.
\newblock Graph implementations for nonsmooth convex programs.
\newblock In V.~Blondel, S.~Boyd, and H.~Kimura, editors, {\em Recent Advances
  in Learning and Control}, Lecture Notes in Control and Information Sciences,
  pages 95--110. Springer-Verlag Limited, 2008.
\newblock \url{http://stanford.edu/~boyd/graph_dcp.html}.

\bibitem[GB10]{cvx}
M.~Grant and S.~Boyd.
\newblock {CVX}: Matlab software for disciplined convex programming, version
  1.21.
\newblock \url{http://cvxr.com/cvx}, May 2010.

\bibitem[GGM11]{ggm}
Bruno Galerne, Yann Gousseau, and Jean-Michel Morel.
\newblock Random phase textures: theory and synthesis.
\newblock {\em IEEE Trans. Image Process.}, 20(1):257--267, 2011.

\bibitem[GM13]{ghoussoub1}
Nassif Ghoussoub and Abbas Moameni.
\newblock A self-dual polar factorization for vector fields.
\newblock {\em Comm. Pure Appl. Math.}, 66(6):905--933, 2013.

\bibitem[GM14]{ghoussoub2}
Nassif Ghoussoub and Bernard Maurey.
\newblock Remarks on multi-marginals symmetric {M}onge-{K}antorovich problems.
\newblock {\em Discrete Cont. Dyn. Syst}, 34(4):1465--1480, 2014.

\bibitem[GS98]{gansw}
Wilfrid Gangbo and Andrzej Swiech.
\newblock Optimal maps for the multidimensional {M}onge-{K}antorovich problem.
\newblock {\em Communications on pure and applied mathematics}, 51(1):23--45,
  1998.

\bibitem[HMM07]{haarala2007globally}
Napsu Haarala, Kaisa Miettinen, and Marko~M M{\"a}kel{\"a}.
\newblock Globally convergent limited memory bundle method for large-scale
  nonsmooth optimization.
\newblock {\em Mathematical programming}, 109(1):181--205, 2007.

\bibitem[HUL96]{urruty1996convex}
J.B. Hiriart-Urruty and C.~Lemar{\'e}chal.
\newblock {\em Convex analysis and minimization algorithms}, volume~1.
\newblock Springer, 1996.

\bibitem[LMfASA78]{lemarechal1978nonsmooth}
C.~Lemar{\'e}chal, R.~Mifflin, and International~Institute for Applied
  Systems~Analysis.
\newblock {\em Nonsmooth optimization}.
\newblock Pergamon Press, 1978.

\bibitem[LO08]{lewis2008behavior}
A.S. Lewis and M.L. Overton.
\newblock Behavior of {BFGS} with an exact line search on nonsmooth examples.
\newblock Technical report, Technical report, Optimization Online, 2008b.
  http://www. optimization-online. org/DB\_FILE/2008/12/2173. pdf, submitted to
  SIAM J. Optimization, 2008.

\bibitem[LO09]{lewis2009nonsmooth}
A.S. Lewis and M.L. Overton.
\newblock Nonsmooth optimization via {BFGS}.
\newblock {\em SIAM Journal of Optimization, submitted for publication}, 2009.

\bibitem[McC97]{mci}
Robert~J. McCann.
\newblock A convexity principle for interacting gases.
\newblock {\em Adv. Math.}, 128(1):153--179, 1997.

\bibitem[MO12]{mo2012discrete}
Quentin M{\'e}rigot and {\'E}douard Oudet.
\newblock Discrete optimal transport: complexity, geometry and applications.
\newblock Preprint, 2012.

\bibitem[Pas12a]{pass1}
Brendan Pass.
\newblock Multi-marginal optimal transport and multi-agent matching problems:
  uniqueness and structure of solutions.
\newblock {\em arXiv preprint arXiv:1210.7372}, 2012.

\bibitem[Pas12b]{pass2}
Brendan Pass.
\newblock On the local structure of optimal measures in the multi-marginal
  optimal transportation problem.
\newblock {\em Calc. Var. Partial Differential Equations}, 43(3-4):529--536,
  2012.

\bibitem[Pow76]{powell1976some}
M.J.D. Powell.
\newblock Some global convergence properties of a variable metric algorithm for
  minimization without exact line searches.
\newblock {\em Nonlinear programming}, 9:53--72, 1976.

\bibitem[PPO14]{papadakis2014optimal}
Nicolas Papadakis, Gabriel Peyr{\'e}, and Edouard Oudet.
\newblock Optimal transport with proximal splitting.
\newblock {\em SIAM Journal on Imaging Sciences}, 7(1):212--238, 2014.

\bibitem[RPDB12]{peyre}
Julien Rabin, Gabriel Peyr{\'e}, Julie Delon, and Marc Bernot.
\newblock Wasserstein barycenter and its application to texture mixing.
\newblock In {\em Scale Space and Variational Methods in Computer Vision},
  pages 435--446. Springer, 2012.

\bibitem[Sch03]{schrijver2003combinatorial}
Alexander Schrijver.
\newblock {\em Combinatorial optimization: polyhedra and efficiency},
  volume~24.
\newblock Springer, 2003.

\bibitem[Vil03]{villani}
C{\'e}dric Villani.
\newblock {\em Topics in optimal transportation}, volume~58.
\newblock AMS Bookstore, 2003.

\bibitem[Vil09]{villani2}
C{\'e}dric Villani.
\newblock {\em Optimal transport: old and new}, volume 338.
\newblock Springer, 2009.

\end{thebibliography}

\end{document}